\documentclass[a4paper, reqno, 12pt]{amsart}
\usepackage{amsthm,amsfonts,amssymb,amsmath,amsxtra,amsrefs,array,dynkin-diagrams,graphicx,caption}
\usepackage{xr-hyper}

\usepackage[outdir=./]{epstopdf}

\usepackage[matrix,tips,frame,color,line,poly]{xy}
\usepackage{verbatim}
\usepackage{tikz}
\usepackage[margin=1.25in]{geometry}
\usepackage{mathrsfs}
\newcommand{\remind}[1]{{\bf ** #1 **}}

\RequirePackage{xspace}
\RequirePackage{etoolbox}
\RequirePackage{varwidth}
\RequirePackage{enumitem}
\RequirePackage{tensor}
\RequirePackage{mathtools}
\RequirePackage{longtable}
\RequirePackage{multirow}

\def\ge{\geqslant}
\def\le{\leqslant}
\def\a{\alpha}
\def\b{\beta}

\def\d{\delta}

\def\e{\epsilon}

\def\s{\sigma}

\def\k{\kappa}

\def\x{\xi}
\def\i{^{-1}}

\def\<{\langle}
\def\>{\rangle}

\newcommand{\BC}{\ensuremath{\mathbb {C}}\xspace}
\newcommand{\C}{\BC} % use this a lot

\newcommand{\BF}{\ensuremath{\mathbb {F}}\xspace}
\newcommand{{\BG}}{\ensuremath{\mathbb {G}}\xspace}

\newcommand{{\BK}}{\ensuremath{\mathbb {K}}\xspace}

\newcommand{\BN}{\ensuremath{\mathbb {N}}\xspace}

\newcommand{\BZ}{\ensuremath{\mathbb {Z}}\xspace}

\newcommand{\CC}{\ensuremath{\mathcal {C}}\xspace}

\newcommand{\CP}{\ensuremath{\mathcal {P}}\xspace}

\newcommand{\CU}{\ensuremath{\mathcal {U}}\xspace}

\newcommand{\Ad}{{\mathrm{Ad}}}

\newcommand{\Ztwo}{\BZ/2\BZ}
\newcommand{\Zg}{\BZ_{\ge 0}}

\newcommand{\GL}{\mathrm{GL}}
\newcommand{\GLdagger}{\mathrm{GL}^\dagger}

\newcommand{\inv}{^{-1}}

\newcommand{\SO}{{\mathrm{SO}}}
\renewcommand{\O}{{\mathrm{O}}}
\newcommand{\Sp}{{\mathrm{Sp}}}

\newcommand{\wt}{\widetilde}

\def\tW{\tilde W}

\def\Gu{[G_u]}  %unipotent classes in G
\def\Gue{[G_u]_e}  %unipotent classes in G, in the image W_e
\def\Gou{[G_{0,u}]}  %unipotent classes in G_0  (can't use 0 nicely in command name, use o instead)
\def\Goue{[G_{0,u}]_e}  %unipotent classes in G_0, in the image of W_e
\def\Du{[D_u]}  %unipotent classes in G
\def\Wc{[W]}    %conjugacy classes in W
\def\Wec{[W_e]} %elliptic conjugacy classes in W
\def\WDec{[W^D_e]} %elliptic conjugacy classes in W
\def\tPminus{\wt{\CP_{-1}}}
\def\tPplus{\wt{\CP_{1}}}
\def\epsilonmax{\epsilon_{\text{max}}}
\def\leW{\preceq_W}

\def\leu{\preceq_u}

\def\Fbartwo{\overline{\BF_2}}
\newtheorem{theorem}{Theorem}
\newtheorem{proposition}[theorem]{Proposition}
\newtheorem{lemma}[theorem]{Lemma}

\newtheorem{corollary}[theorem]{Corollary}

\theoremstyle{definition}

\newtheorem{example}[theorem]{Example}

\newtheorem{remark}[theorem]{Remark}

\numberwithin{equation}{section}
\numberwithin{theorem}{section}

\setitemize[0]{leftmargin=*,itemsep=\the\smallskipamount}
\setenumerate[0]{leftmargin=*,itemsep=\the\smallskipamount}

\renewcommand{\to}{%
   \ifbool{@display}{\longrightarrow}{\rightarrow}%
   }
\let\shortmapsto\mapsto
\renewcommand{\mapsto}{%
   \ifbool{@display}{\longmapsto}{\shortmapsto}%
   }
\newlength{\olen}
\newlength{\ulen}
\newlength{\xlen}
\newcommand{\xra}[2][]{%
   \ifbool{@display}%
      {\settowidth{\olen}{$\overset{#2}{\longrightarrow}$}%
       \settowidth{\ulen}{$\underset{#1}{\longrightarrow}$}%
       \settowidth{\xlen}{$\xrightarrow[#1]{#2}$}%
       \ifdimgreater{\olen}{\xlen}%
          {\underset{#1}{\overset{#2}{\longrightarrow}}}%
          {\ifdimgreater{\ulen}{\xlen}%
             {\underset{#1}{\overset{#2}{\longrightarrow}}}
             {\xrightarrow[#1]{#2}}}}%
      {\xrightarrow[#1]{#2}}
   }
\makeatother
\newcommand{\xyra}[2][]{%
   \settowidth{\xlen}{$\xrightarrow[#1]{#2}$}%
   \ifbool{@display}%
      {\settowidth{\olen}{$\overset{#2}{\longrightarrow}$}%
       \settowidth{\ulen}{$\underset{#1}{\longrightarrow}$}%
       \ifdimgreater{\olen}{\xlen}%
          {\mathrel{\xymatrix@M=.12ex@C=3.2ex{\ar[r]^-{#2}_-{#1} &}}}%
          {\ifdimgreater{\ulen}{\xlen}%
             {\mathrel{\xymatrix@M=.12ex@C=3.2ex{\ar[r]^-{#2}_-{#1} &}}}
             {\mathrel{\xymatrix@M=.12ex@C=\the\xlen{\ar[r]^-{#2}_-{#1} &}}}}}%
      {\mathrel{\xymatrix@M=.12ex@C=\the\xlen{\ar[r]^-{#2}_-{#1} &}}}%
   }
\makeatletter
\newcommand{\xla}[2][]{%
   \ifbool{@display}%
      {\settowidth{\olen}{$\overset{#2}{\longleftarrow}$}%
       \settowidth{\ulen}{$\underset{#1}{\longleftarrow}$}%
       \settowidth{\xlen}{$\xleftarrow[#1]{#2}$}%
       \ifdimgreater{\olen}{\xlen}%
          {\underset{#1}{\overset{#2}{\longleftarrow}}}%
          {\ifdimgreater{\ulen}{\xlen}%
             {\underset{#1}{\overset{#2}{\longleftarrow}}}
             {\xleftarrow[#1]{#2}}}}%
      {\xleftarrow[#1]{#2}}
   }
\newcommand{\isoarrow}{%
   \ifbool{@display}{\overset{\sim}{\longrightarrow}}{\xrightarrow\sim}%
   }

\begin{document}
\thispagestyle{plain}

\author{Jeffrey Adams}
\author{Xuhua He}
\author{Sian Nie}

%use one or the other of these:

% using \documentclass[a4paper, reqno, 10pt]{article}:
%\affil{Department of Mathematics, University of Maryland, jda@math.umd.edu}
%\affil{Department of Mathematics, University of Maryland, xuhuahe@math.umd.edu}
%\affil{Institute of Mathematics, Academy of Mathematics and Systems Science, Chinese Academy of Sciences, 100190, Beijing, China
%  , niesian@amss.ac.cn}

% %\documentclass[a4paper, reqno, 12pt]{amsart}
\address{Department of Mathematics, University of Maryland, jda@math.umd.edu}
\address{The Institute of Mathematical Sciences and Department of Mathematics, The Chinese University of Hong Kong, Shatin, N.T., Hong Kong, xuhuahe@gmail.com}
\address{Institute of Mathematics, Academy of Mathematics and Systems Science, Chinese Academy of Sciences, 100190, Beijing, China, niesian@amss.ac.cn}

%\date{}                     %% if you don't need date to appear

\title[Partial orders]{Partial orders on conjugacy classes in the Weyl group and on unipotent conjugacy classes}

\date{\today}

\begin{abstract}
Let $G$ be a reductive group over an algebraically closed field and let $W$ be its Weyl group. In a series of papers, Lusztig introduced a map from the set $[W]$ of conjugacy classes of $W$ to the set $[G_u]$ of unipotent classes of $G$. This map, when restricted to the set of elliptic conjugacy classes $[W_e]$ of $W$, is injective. In this paper, we show that Lusztig's map $[W_e] \to [G_u]$ is order-reversing, with respect to the natural partial order on $[W_e]$ arising from combinatorics and the natural partial order on $[G_u]$ arising from geometry. 
\end{abstract}


\maketitle

\tableofcontents

\section*{Introduction}

\subsection{Lusztig's map}
Let $G$ be a connected reductive group over an algebraically closed
field $\BF$ and let $W$ be the Weyl group of $G$. Let $\Gu$ be the set of
unipotent conjugacy classes in $G$ and $\Wc$ be the set of conjugacy
classes of $W$. Lusztig defined a surjective map
$\Phi: \Wc \to \Gu$ (\cite[Theorem 0.4]{L1}).
This construction was generalized to twisted conjugacy classes in \cite{L3}.

Roughly speaking, the map $\Phi$ is constructed as follows. 
Let $\CC \in \Wc$ and $w \in \CC_{\min}$ be a minimal length element of
$\CC$. We look at the intersection of  the Bruhat double coset $B w B$ with
unipotent conjugacy classes and we select the minimal
unipotent class which gives a nonempty intersection. It is pointed out
in \cite[\S 0.1]{L1} that ``the fact that the procedure actually works
is miraculous''.  

In this paper we are concerned with the set of elliptic conjugacy classes
$\Wec \subset \Wc$.  The restriction of $\Phi$ to $\Wec$ is  injective,
and the image contains all distinguished unipotent
conjugacy classes (\cite[Proposition 0.6]{L1}).

The set $\Gu$ has the natural partial ordering by closure
relations, denote $\leu$. On the other hand the
second-named author \cite{He07} (see also \cite[\S
1.10.3]{He-CDM}) introduced a partial order on $\Wec$,
induced from the Bruhat order on  minimal length elements of the
elliptic conjugacy classes of $W$, which we denote $\leW$.

It is a natural question to consider how Lusztig's map
behaves with respect to these partial orders.
Dudas, Michel and the second-named author
\cite{DHM} conjectured that $\Phi$ gives an order-reversing bijection
from $\Wec$ to $\Phi(\Wec) \subset \Gu$. Michel \cite{DHM} verified the exceptional groups by computer. 

In \cite[Conjecture 3.7]{DM}, Dudas and Malle conjectured that a
similar result holds for twisted type $A$. In \cite[Proposition 3.11
\& 3.14]{DM} they verified the conjecture for some special family of
twisted elliptic conjugacy classes of type $A$. It is also verified in
\cite{DM} by computer that the conjecture holds for twisted $A_n$ with
$n \le 10$. The compatibility of the partial orders is used in
\cite[\S 5]{DM} to study the decomposition numbers of the unipotent
$l$-blocks of finite unitary groups.

Our first main result is  that $\Phi$ is order reversing in the following sense.

\begin{theorem}
\label{t:main}
Let $\CC, \CC' \in \Wec$. Then $\CC \leW\CC'$ if and only if
$\Phi(\CC') \leu\Phi(\CC)$.
\end{theorem}

This holds also in the twisted setting.

\subsection{A similar phenomenon for affine Weyl groups}

Before we discuss the strategy towards Theorem \ref{t:main}, we make a
short digression and discuss a similar phenomenon.
For simplicity, we only discuss  split groups here but the result
holds in general.

Let $G$ be a connected reductive group. The Frobenius morphism $\s$ of
$\overline{\mathbb F_q}$ over $\mathbb F_q$ induces a Frobenius
morphism $\s$ on $G(\overline{\mathbb F_q}((t)))$. Let $B(G)$ be the
set of $\s$-twisted conjugacy classes on
$G(\overline{\mathbb F_q}((t)))$. Let $\tW$ be the Iwahori-Weyl group
and let $[\tW]$ be the set of conjugacy classes of $\tW$. Inside
$[\tW]$, there is a special subset, the set of straight
conjugacy classes, which we denote by $[\tW_{str}]$. It is
proved in \cite{He-Ann} that there is a natural bijection
$\tilde \Phi: [\tW_{str}] \to B(G)$, which is induced from
any lifting $\tW \to G(\overline{\mathbb F_q}((t)))$.

The closure relation gives the partial order on $B(G)$. There is also
a partial order on $[\tW_{str}]$ induced from the Bruhat
order on $\tW$, similar to the definition of the partial order
$\leW$ on $\Wec$ we discussed earlier. It is proved in
\cite[Theorem B]{He-KR} that these two partial orders coincide via the
bijection $\tilde \Phi: [\tW_{str}] \to B(G)$.

The proof uses the reduction method of Deligne and Lusztig \cite{DL},
some remarkable combinatorial properties on the straight conjugacy
classes \cite{HN2} and a deep result in arithmetic geometry, the
purity theorem for the Newton stratification associated with
$F$-crystal obtained by de Jong-Oort \cite{JO}, Hartl-Viehmann
\cite{HV}, Viehmann \cite{Vi} and Hamacher \cite{Ha}.

\subsection{Difference between the finite and affine cases}

Although in both finite and affine cases, we compare the partial
orders arising from combinatorics and from geometry, there are some
essential differences between the two cases we discussed above.

First, in the affine case, we do not consider the conjugation action
on the loop groups, but the Frobenius-twisted conjugacy classes
instead. Although the study of the Frobenius-twisted conjugacy classes
are quite involved, it is, in some sense, simpler than the unipotent
conjugacy classes. Second, the construction of the map from conjugacy
classes of Weyl groups to the (twisted) conjugacy classes of reductive
groups in the finite and affine case are quite different. In the
affine case, the map is induced from any lifting
$\tW \to G(\overline{\mathbb F_q}((t)))$, while in the finite case the
construction is rather a miracle. Finally, in the affine case, the map
preserves the partial orders, while in the finite case, as we show in
this paper, the map reverses the partial orders.

\subsection{The strategy}

Now we discuss the strategy towards the proof of Theorem
\ref{t:main}. For exceptional groups, we verify the statement by
computer in Section \ref{6-exc}. For classical groups, the elliptic conjugacy classes are parametrized by certain partitions. For such  partition $\alpha$, we denote by $\CC_\alpha$ the corresponding elliptic conjugacy class. We show that 
$$
\Phi(\CC_\alpha)\leu\Phi(\CC_\beta)\Leftrightarrow \alpha\le\beta\Leftrightarrow \CC_\beta\leW\CC_\alpha.
$$

Note that the unipotent classes in the classical groups are associated
to certain partitions. However, the partitions associated to elliptic
conjugacy classes and the unipotent classes of the same classical
group, are usually partitions of different integers. For example, in
type $B_n$, the elliptic conjugacy classes correspond to partitions of
$n$ while the unipotent conjugacy classes correspond to certain
partitions of $2n+1$. The map from partitions of $n$ to partitions of
$2n+1$ induced from Lusztig's map in characteristic $0$ (and any
characteristic $\neq 2$) are rather complicated; however, in
characteristic $2$ the map is  rather simple and is essentially the map
$\alpha \mapsto (2 \alpha, 1)$. We then use the Lusztig-Spaltenstein
map from the set of unipotent classes in characteristic $0$ to the set
of unipotent classes in characteristic $2$ to reduce the
statement
\[\tag{*}\Phi(\CC_\alpha)\leu\Phi(\CC_\beta)\Leftrightarrow
  \alpha\le\beta\] in characteristic $0$ (and any characteristic
$\neq 2$) to the statement (*) in characteristic $2$, which is obvious
since the map $\alpha \mapsto (2 \alpha, 1)$. For other classical
groups, the statement
$\Phi(\CC_\alpha)\leu\Phi(\CC_\beta)\Leftrightarrow \alpha\le\beta$ is
verified in the same way. This is done in section \ref{3-cl}.

We then verify the statement 
\[\tag{**} \alpha\le\beta\Leftrightarrow \CC_\beta\leW\CC_\alpha\] for classical groups. 

By definition, $\CC_{\beta} \leW \CC_{\alpha}$ if and only if there exist minimal length elements $w_\beta \in \CC_{\beta}$ and $w_\a \in \CC_{\a}$ such that $w_{\beta} \le w_\a$. In \cite{He07}, the second-named author constructed explicit minimal length representatives for any elliptic conjugacy class of Weyl groups of classical type. If $\alpha \le \beta$, then we have the desired relation between those minimal length representatives with respect to the Bruhat order. This proves the $\Rightarrow$ direction of the statement (**) for classical groups. 

The $\Leftarrow$ direction of the statement (**) is more involved. The difficulty is that there are many minimal length elements in a given elliptic conjugacy class and the number is unbounded as the rank of the group becomes larger. To overcome the difficulty, we use the explicit description of the Bruhat order for the Weyl groups of classical type (see \cite{bb}).  For example, for any elements $w$ in the Weyl group of type $B_n$ and $-n \le i, j \le n$, we may associate a nonnegative integer $w[i, j]$. Then $w \le w'$ if and only if $w[i, j] \le w'[i, j]$ for all $-n \le i, j \le n$. So there are $4 n^2$ inequalities to check. Fortunately, to study $\CC_{\beta} \leW \CC_{\alpha}$, one only needs to investigate $l$ inequalities among all the $4 n^2$ inequalities, where $l$ is the number of parts in the partition $\beta$. And these $l$ inequalities imply that $\alpha \le \beta$. This is done in section \ref{s:classical_proof}.

\smallskip

\noindent {\bf Acknowledgements:} The idea that the partial orders on the conjugacy classes of Weyl groups and the unipotent classes of algebraic groups might be related was initiated in the private conversation of the second-named author with Olivier Dudas and Jean Michel. The explicit description of the Bruhat order for classical groups was pointed out to us by Thomas Lam. We also thank George Lusztig and Zhiwei Yun for helpful discussions. 

\section{Main result}

\subsection{Preliminary} Let $G$ be an affine algebraic group over an algebraically closed field $\BF$ of characteristic $p \ge 0$ such that the identity component $G^0$ of $G$ is reductive. Let $T$ be a maximal torus of $G^0$ and $B \supset T$ be a Borel subgroup of $G^0$. Let $W^0=N_{G^0}(T)/T$ be the Weyl group of $G^0$ and let $W=N_G(T)/T$ be the (extended) Weyl group of $G$. The length function $\ell$ on $W^0$ extends in a unique way to a length function on $W$, which we still denote by $\ell$. Let $S \subset W^0$ be a set of simple reflections.

Let $D$ be a connected component of $G$ and $W^D=(N_G(T) \cap D)/T$ be a left/right $W^0$-coset of $W$. By \cite[Section 1.4]{L03}, $W^D$ contains a unique element $\e_D$ of length $0$. The conjugation action of $\e_D$ on $W$ is a length-preserving automorphism. Let $[W]$ be the set of $W^0$-conjugacy classes of $W$ and $[W^D]$ be the set of $W^0$-conjugacy classes of $W$ that intersect $D$. An element $w \in W^D$ (or its $W^0$-conjugacy class $C$ in  $W$) is said to be {\it elliptic} if for any $J \subsetneqq S$ with $\Ad(\e_D)(J)=J$, we have $C \cap W_J \e_D=\emptyset$. Let $\WDec$ the set of elliptic $W^0$-conjugacy classes of $W$ that intersect $W^D$.

From now on, we assume that $D$ contains a unipotent element of $G$. Let $[D_u]$ be the set of $G^0$-conjugacy classes of $D$ which are unipotent. In \cite{L1} and \cite{L3}, Lusztig introduced a  map $\Phi: [W^D] \to [D_u]$. It is proved in loc. cit. that

\begin{itemize}
\item The map $\Phi: [W^D] \to [D_u]$ is surjective.

\item The restriction to elliptic conjugacy classes $\Phi_e: \WDec \to [D_u]$ is injective.
\end{itemize}

\subsection{Partial orders} We define a partial order on unipotent classes by the closure relations as usual:
$\CC\leu \CC'$ if $\CC\subset \overline{\CC'}$.

Now we recall the partial order on $\WDec$ introduced in \cite[\S 4.7]{He07}. Let $\CC, \CC' \in \WDec$. We denote by $\CC_{\min}$ (respectively $\CC'_{\min}$) the set of minimal length elements in $\CC$ (respectively $\CC'$). Then the following conditions are equivalent:

\begin{enumerate}
  \item For some $w \in \CC_{\min}$, there exists $w' \in \CC'_{\min}$ such that $w' \le w$;
  \item For any $w \in \CC_{\min}$, there exists $w' \in \CC'_{\min}$ such that $w' \le w$.
\end{enumerate}

If these conditions are satisfied, then we write $\CC' \leW \CC$. By the equivalence of the conditions (1) and (2) above, the relation $\leW$ is transitive. This gives a natural partial order on the set $\WDec$.

By \cite[Corollary 4.5]{He07}, $\CC'_{\min}$ is the set of minimal elements in $\CC'$ with respect to the Bruhat order of $W$. Thus the conditions (1) and (2) above are also equivalent to the following conditions:

\begin{enumerate}
\setcounter{enumi}{2}
  \item For some $w \in \CC_{\min}$, there exists $w' \in \CC'$ such that $w' \le w$;
  \item For any $w \in \CC_{\min}$, there exists $w' \in \CC'$ such that $w' \le w$.
\end{enumerate}

The condition (3) will be used to study the partial order on $\WDec$ for exceptional groups.

Now we state the main theorem of the paper.

\begin{theorem}\label{main}
The map $\Phi_e: \WDec \to [D_u]^{op}$ gives a bijection from the poset $\WDec$ to its image. Here $[D_u]^{op}$ is the same as $[D_u]$ as a set, but with reversed partial order.

In other words let $\CC, \CC' \in \WDec$. Then $\CC' \leW \CC$ if and only if $\Phi(\CC) \leu \Phi(\CC')$.
\end{theorem}

\subsection{Reduction to almost simple groups} In this subsection, we show that to prove Theorem \ref{main}, it suffices to consider the case where $G^0$ is simple. The reduction procedure is the similar to \cite[\S 1.5--1.11]{L3}.

First, we may replace $G$ by the subgroup generated by $D$. Next, let $G'=G/Z(G^0)$ and $\pi: G \to G'$ be projection map. As the Weyl groups of $G$ and $G'$ are naturally identical, Theorem \ref{main} holds for $(G, D)$ if and only if it holds for $(G', D')$, where $D'=\pi(D)$.

Now we may assume that $G^0$ is semisimple and simply connected. We write $G^0$ as $G^0=G_1 \times \ldots \times G_k$, where each $G_i \neq \{1\}$ is a minimal closed connected normal subgroup of $G$. For any $i$, let $G'_i=G/(G_1 \times \ldots \times \hat G_i \times \ldots \times G_k)$ and $D_i$ be the image of $D$ in $G'_i$. Let $G'=G'_1 \times \ldots \times G'_k$. We may then identify $G$ with a closed subgroup of $G'$ with the same identity component. Under this identification, $D$ becomes $D_1 \times \ldots \times D_k$. Let $W'$ be the extended Weyl group of $G'$ and $W'_i$ be the extended Weyl group of $G'_i$. Then $W'=W'_1 \times \ldots \times W'_k$ and we may identify $W^D$ with $(W'_1)^{D_1} \times \ldots \times (W'_k)^{D_k}$. Under this identification, $$\WDec=[(W'_1)^{D_1}_e] \times \ldots \times [(W'_k)^{D_k}_e]$$ and $\leW$ on $\WDec$ coincides with $\preceq_{W, 1} \times \ldots \times \preceq_{W, k}$ on $[(W'_1)^{D_1}_e] \times \ldots \times [(W'_k)^{D_k}_e]$. By \cite[\S 1.8]{L3}, $\Phi_e$ on $\WDec$ coincides with $\Phi_{e, 1} \times \ldots \times \Phi_{e, k}$ on $[(W'_1)^{D_1}_e] \times \ldots \times [(W'_k)^{D_k}_e]$. Thus if Theorem \ref{main} holds for each $(G'_i, D_i)$, then it holds for $(G, D)$.

Now we may assume that $G^0$ is semisimple, simply connected and that $G$ has no nontrivial closed connected normal subgroups. By \cite[\S 1.9]{L3}, $G^0=H_1 \times \ldots \times H_{m}$, where $H_i$ are connected, simply connected, almost simple, closed subgroups of $G^0$ and there exists a $c \in D$ such that $H_i=c^i H_1 c^{-i}$ for $0 \le i \le m-1$ and $c^m H_1 c^{-m}=H_1$. Let $G'$ be the subgroup of $G$ generated by $H_1$ and $c^m$ and $D'=c^m H_0$ be a connected component of $G'$. By \cite[\S 1.9]{L3}, we may identify $[D_u]$ with $[D'_u]$ and $[W^D]$ with $[(W')^{D'}]$ and under this identification, we have the following commutative diagram
\[
\xymatrix{
[(W')^{D'}] \ar@{=}[r] \ar[d]_-{\Phi'} & [W^D] \ar[d]^-{\Phi} \\
[D'_u] \ar@{=}[r] & [D_u].
}
\]
Thus Theorem \ref{main} holds for $(G, D)$ if and only if it holds for $(G', D')$.

Therefore, to prove Theorem \ref{main}, it suffices to consider the cases where $G^0$ is almost simple.

\section{Unipotent conjugacy classes of classical groups}
\label{s:classical}

In this section, we recollect some facts on the unipotent conjugacy
classes of classical groups, over an
algebraically closed field $\BF$ of any characteristic. We follow \cite{Spa}*{Section I.2}.

Suppose $V$ is a finite-dimensional vector
space over $\BF$.
We define a disconnected group containing $\GL(V)$ $(\dim(V)\ge 3)$ as in
\cite{Spa}*{Section I.2.7}.
Define $\GLdagger(V)=G_0\cup G_1$ where $G_0=GL(V)$ and $G_1$ is
the set of non-singular bilinear forms $\phi:V\times V\rightarrow \BF$.
We define the product structure on $G$ as follows. The product on
$G_0$ is the usual one. If $g\in G_0,\phi\in G_1$ then $(g\phi)(v,w)=\phi(g\inv v,w)$
and $(\phi g)(v,w)=\phi(v,gw)$. If $\phi,\psi\in G_1$ then $\phi\psi$ is the unique element of $G_0$ satisfying
$\phi((\phi\psi)v,w)=\psi(w,v)$.

It is easy to see $\GLdagger(V)$ is a group and
$\GLdagger(V)/\GL(V)\simeq\Ztwo$.  Also $\phi$ is symmetric if and
only if $\phi^2=1$, so there is a unique $\GL(V)$-conjugacy class of
such elements. Choose a basis of $V$ and identify $\GL(V)$ with
$\GL(n)$, and set $\delta(v,w)=v\cdot w$. Then $\delta^2=1$ and
$\delta g\delta=\,^tg\inv$ for $g\in \GL(V)$.

By a {\it classical group} we mean one of the groups $\GL(n), \Sp(2n),\SO(n), \O(n)$ or $\GLdagger(n)$.
The groups $\GL(n)$ and $\Sp(2n)$ are connected.
The identity component of $\O(n)$ is $\SO(n)$, and $\O(n)/\SO(n)$ is trivial if $p=2$, and has order $2$ otherwise.

Let $\CP(n)$ be the set of partitions of $n$. We write a partition of $n$ as
$\alpha=(\alpha_1,\ldots, \alpha_\ell)$ with $\alpha_1\ge \dots\ge \alpha_\ell \ge 0$ and $\sum\alpha_i=n$.
On occasion we do allow some $\a_i$ to be $0$, i.e., we regard $(\a_1, \ldots, \a_\ell)$ and $(\a_1, \ldots, \a_\ell, 0)$ to be the same partition. 
We define the standard partial order on partitions of the same integer $n$: $\alpha\le\beta$ if for all $k$,
$\sum_{i=1}^k\alpha_i=\sum_{i=1}^k\beta_k$.

We identify partitions with Young diagrams, and define the transpose
partition as usual. If $\alpha=(\alpha_1,\dots,\alpha_\ell)$ is a
partition we let $\alpha^*=(\alpha_1^*,\dots, \alpha_m^*)$ be the
transpose partition. In particular $\alpha_1^*$ is the number of rows
of $\alpha$.

Given a partition $\alpha$ define the multiplicity function $m_\alpha:\Zg\rightarrow \Zg$ as usual:
$m_\alpha(k)=|\{j\mid \alpha_j=k\}|$. In particular $m_\alpha(k)=0$ for $k=0$ or $k>\alpha_1$.
For $\k=\pm 1$ let
$$
\CP_\k(n)=\{\alpha \in \CP(n)\mid m_\alpha(i)\text{ is even if }  (-1)^i=\kappa\}
$$
Let $\CP(n)_0$ be the partitions of $n$ with an even number of parts, i.e. with $\a_1^*$ even and $\CP(n)_1$ be the partitions of $n$ with an odd number of parts, i.e. with $\a_1^*$ odd.
Let $\CP(n)^{\text{odd}}$  be the partitions consisting only of odd parts.

\subsection{Unipotent classes in good characteristic}
\label{s:good}

If $G=\GLdagger(n), \Sp(2n)$ or $\O(n)$ then the characteristic $p=2$ is said to be {\it bad} (for $G$).
Otherwise, including all groups in characteristic $0$, the characteristic is said to be {\it good}.

Suppose the characteristic of $\BF$ is good. Then every unipotent
element of $G$ is contained in $G^0$, and
the unipotent classes are parametrized as follows.

\begin{enumerate}
\item $\GL(n)$ or $\GLdagger(n)$: $\CP(n)$;
\item $\O(2n+1)$: $\CP_1(2n+1)$;
\item $\Sp(2n)$: $\CP_{-1}(2n)$;
\item $\O(2n)$: $\CP_1(2n)$.
\end{enumerate}

We also consider the unipotent conjugacy classes of $\SO(n)$.  If $n$
is odd the unipotent conjugacy class of $\SO(n)$ and $O(n)$ are in
bijection, and the same holds for $\SO(2n)$ if $n$ is odd.  If $n$
is even the unipotent $\SO(2n)$-conjugacy classes in $\SO(2n)$ are parametrized
by $\CP_1(2n)$, except that every partition with only even parts
corresponds to two classes; there are $p(n/2)$ of these classes where
$p$ is the partition function.
The $\O(2n)$ orbits which split into two $\SO(2n)$ orbits do not arise
in the image of Lusztig's map applied to elliptic conjugacy classes, so we do not need
to distinguish these two classes.

We write $\CU_{\alpha}$ for the unipotent class parametrized by a partition $\alpha$.
Then the partial order on unipotent conjugacy classes is given by the partial order on partitions:
$\CU_{\a} \leu \CU_{\b}$ if and only if $\a \le \b$.

\subsection{Unipotent conjugacy classes in bad characteristic}
\label{s:bad}

Suppose $G$ is a classical group and the characteristic $p$ of $\BF$ is
bad, in particular $p=2$.  There is a bijective algebraic group homomorphism
from $\Sp(2n)$ to $SO(2n+1)$ (although the inverse is not algebraic),
which induces a bijection of unipotent classes. Also
$SO(2n+1)=O(2n+1)$ so we do not need to consider these groups.

\subsubsection{The cases $G=Sp(2n)$  and $O(2n)$}

Consider a set $\{\omega,0,1\}$ where $\omega$ is a formal element,
satisfying $\omega<0<1$.  For $n\ge 1$ define
$\tPminus(n)$ to be the set of pairs $(\alpha,\epsilon)$, where
\begin{enumerate}
\item $\alpha\in \CP_{-1}(2n)$ (i.e. odd rows have even multiplicity);
\item $\epsilon:\Zg\rightarrow\{\omega,0,1\}$.
\end{enumerate}
The function $\epsilon$ is required to satisfy, for all $i\ge 0$:
$$
\epsilon(i)=
\begin{cases}
1,& \text{ if } i=0, G=Sp(n);\\
0,& \text{ if } i=0, G=O(n);\\
\omega,& \text{ if } i\text{ odd};\\
\omega,& \text{ if } i>0,\,m_\alpha(i)=0;\\
1,& \text{ if } i>0\text{ even},\, m_\alpha(i)\text{ odd};\\
0\text{ or }1,& \text{ if } i>0\text{ even},\, m_\alpha(i)>0 \text{ even}.
\end{cases}
$$
Note that $\tPminus(n)$ is empty if $n$ is odd, and $\epsilon(i)$ is determined by $\alpha$ except for even rows of even multiplicity.

\begin{proposition}
  If $p=2$ the unipotent conjugacy classes in $Sp(2n)$ or $O(2n)$ are
  in bijection with $\tPminus(2n)$.
\end{proposition}

We write $\CU_{\alpha,\epsilon}$ for the unipotent class parametrized by $(\alpha,\epsilon)$.

The map from unipotent classes to $\tPminus(n)$ is defined as
follows.  First assume $G=Sp(2n)$, and let $\langle\,,\,\rangle$ be
the symplectic form defining $G$.  We embed
$\phi:\Sp(2n)\rightarrow G^*=GL(2n)$ as usual.  If $g\in G$ is
unipotent then $\phi(g)\in G^*$ is unipotent, and so
corresponds to a partition  $\alpha$ of $2n$; it is easy to see $\alpha\in\CP_{-1}(2n)$.

Suppose $i>0$ is even and $m_\alpha(i)>0$. Set $\epsilon(i)=0$ if
$\langle (g-1)^{i-1}v,v\rangle=0$ for all $v\in \text{ker}(g-1)^i$, and $\epsilon(i)=1$ otherwise.
Together with the conditions above this defines $\epsilon$ uniquely.

Next, if $G=O(2n)$ we note that every unipotent conjugacy class in
$Sp(2n)$ intersects $O(2n)$ in a unique conjugacy class, and this
defines a bijection between unipotent conjugacy classes in $Sp(2n)$ and $O(2n)$.

Write $\CU_{\alpha,\epsilon}$ for
the unipotent conjugacy class associated to
$(\alpha,\epsilon)\in \tPminus(2n)$.

In the case of $G=O(2n)$ we need to distinguish between those
$G$-conjugacy classes contained in $\SO(2n)$ and those which are not.

Define $\tPminus(2n)_0\subset \tPminus(2n)$ to be the pairs $(\alpha,\epsilon)$ such that $\alpha_1^*$ is even, and
set $\tPminus(2n)_1=\tPminus(2n)\backslash \tPminus(2n)_0$.

\begin{lemma}
Suppose $(\alpha,\epsilon)\in \tPminus(2n)$.
Then $\CU_{\alpha,\epsilon}\subset \SO(2n)$ if and only if $(\alpha,\epsilon)\in\tPminus(2n)_0$.
\end{lemma}

Thus $\tPminus(2n)_0$ (respectively $\tPminus(2n)_1$) is in bijection with the unipotent  $\O(2n)$ conjugacy classes in $\SO(2n)$ (respectively
$\O(2n)\backslash\SO(2n)$).

Finally we consider
unipotent $\SO(2n)$-conjugacy classes.
If $(\alpha,\epsilon)\in \tPminus(2n)_0$ then
 $\CU_{\alpha,\epsilon}\subset \SO(2n)$ is the union of two $\SO(2n)$-conjugacy classes if
 for all $i$, $\alpha_i$ and $m_\alpha(i)$ are even and $\epsilon(i)=0$. Otherwise
 $\CU_{\alpha,\epsilon}\subset\SO(2n)$ is a single $\SO(2n)$-conjugacy class.
Again the $\O(2n)$ orbits which split into two $\SO(2n)$ orbits do not arise
 in the image of Lusztig's map applied to elliptic elements, so we do not need
 to distinguish these two classes.

\subsubsection{The case $G=\GLdagger(n)$} Recall every characteristic for $\GL(n)$ is good, and the unipotent
classes for $\GL(n)$ are parametrized by partitions of $n$. Now we
consider $\GLdagger(n)$ ($n\ge 3$). Recall we write
$\GLdagger(n)=G_0\cup G_1$ where $G_0=\GL(n)$ and $G_1$ is the set of
non-singular bilinear forms.

\begin{lemma}
The unipotent conjugacy classes of $\GLdagger(n)$ which are contained in
$\GL(n)$ are in bijection with the unipotent conjugacy classes of $\GL(n)$.
\end{lemma}

In other words if $\CU\subset \GL(n)$ is a unipotent conjugacy class for $\GLdagger(n)$
then it is a single $\GL(n)$-orbit.

So consider the unipotent conjugacy classes of $\GLdagger(n)$ in $\GLdagger(n)\backslash \GL(n)$.

Define $\tPplus(n)$ to be the set of pairs $(\alpha,\epsilon)$,
where
\begin{enumerate}
\item $\alpha\in \CP_{1}(n)$ (i.e. even rows have even multiplicity);
\item $\epsilon:\Zg\rightarrow\{\omega,0,1\}$.
\end{enumerate}
The function $\epsilon$ is required to satisfy, for all $i\ge 0$:
$$
\epsilon(i)=
\begin{cases}
\omega,& \text{ if } i\text{ even};\\
\omega,& \text{ if } m_\alpha(i)=0;\\
1,& \text{ if } i\text{ odd}, m_\alpha(i)\text{ odd};\\
0\text{ or }1,& \text{ if } i\text{ odd},\, m_\alpha(i)>0\text{ even}.
\end{cases}
$$

\begin{proposition}
If $p=2$, the unipotent conjugacy classes of $\GLdagger(n)$ in $\GLdagger(n)\backslash \GL(n)$ are parametrized by $\tPplus(n)$.
\end{proposition}

We write $\CU_{\alpha,\epsilon}$ for the unipotent class parametrized by $(\alpha,\epsilon)$.

The map from unipotent classes in $G_1$ to parameters is defined as
follows.  Define the map $S:G_1\rightarrow G_0$ to be
$S(\phi)=\phi^2$. If $\phi$ is unipotent then so is $S(\phi)$, and
therefore $S(\phi)$ defines a partition $\alpha$ of $n$, and it is easy to see $\alpha\in\CP_1(n)$.
Suppose $i$ is odd and $m_\alpha(i)>0$ is even. Then define
$\epsilon(i)=0$ if $f(v,(g-1)^{i-1}v)=0$ for all $v\in \text{ker}(g-1)^i$, and $\epsilon(i)=1$ otherwise.

\subsection{Closure Relations}
\label{s:closure}

We now describe the closure relations on unipotent classes. As discussed
in Section \ref{s:good} if the characteristic is good then the closure
relations on classes are given by the order relation on partitions.
So assume $p=2$ and $G=\GLdagger$, $Sp(2n)$ or $O(2n)$.

We define a partial order on the sets
$\tPminus(2n)$ and $\tPplus(n)$ defined in the previous section.
Suppose $(\alpha,\epsilon)$ and $(\beta,\delta)$ are elements of one of these sets.
We say $(\alpha,\epsilon)\le (\beta,\delta)$ if

\begin{enumerate}
\item $\alpha\le \beta$ \quad $(\Leftrightarrow \alpha^*\ge \beta^*)$;
\end{enumerate}
and for all $k\ge 1$:
\begin{enumerate}
\item[(2)] $\left(\sum_{i=1}^k\beta_i^*\right)-\text{max}(\delta_k,0)\le\left(\sum_{i=1}^{k}\alpha_i^*\right)-\text{max}(\epsilon_k,0)$;
\item[(3)] If $\sum_{i=1}^k\alpha_i^*=\sum_{i=1}^k\beta_i^*$ and $\alpha^*_{k+1}-\beta^*_{k+1}$ is odd then $\delta_k\ne 0$.
\end{enumerate}

\begin{proposition}
  \label{p:order}
Suppose $p=2$ and $G=\GLdagger(n), Sp(2n)$ or $O(2n)$.
Suppose $(\alpha,\epsilon),(\beta,\delta)$ are  both in
$\tPplus(n), \tPminus(2n)$, or $\tPminus(2n)$, respectively.
In the case of $O(2n)$ assume they are both in $\tPminus(2n)_0$ or $\tPminus(2n)_1$.

Then
$\CU_{\alpha,\epsilon}\leu \CU_{\beta,\delta}$ if and
only if $(\alpha,\epsilon)\le (\beta,\delta)$.
\end{proposition}

\begin{corollary}
Fix a partition $\alpha$ of the appropriate type for $G$. Then the set $\{\CU_{\alpha,\epsilon}\}$ (as $\epsilon$ varies)
has a unique maximal element.
\end{corollary}

\begin{proof}
Recall $\epsilon(i)\in\{\omega,0,1\}$, and is determined by $\alpha$
except in some cases when it can be $0$ or $1$; for $\epsilonmax$
always choose $1$ whenever there is a choice.

Explicitly, define:
\[
  \epsilonmax^\alpha(i)=
  \begin{cases}
    1, & \text{ if } \quad i>0\text{ odd},\, m_\alpha(i)>0\text{ even and }\alpha\in\tPplus(n);\\
    1, & \text{ if }  \quad i>0\text{ even},\, m_\alpha(i)>0\text{ even and } \alpha\in\tPminus(2n).
  \end{cases}
\]
This determines $\epsilonmax^\alpha$ completely, and the result is immediate from Proposition \ref{p:order}.
\end{proof}

\section{Lusztig's map for classical groups}\label{3-cl}
\label{s:Lclassical}
\subsection{The function $\psi$}
We recall a function defined in  \cite[\S 1.6]{L1}.

Suppose $\alpha=(\alpha_1,\dots,\alpha_\ell)$ is a partition. Define a function
$$\psi_\alpha:\{i\mid 1\le i\le \ell\}\rightarrow\{-1,0,1\}$$  as follows.
Set $\alpha_0=\alpha_{\ell+1}=0$.

\begin{enumerate}

\item If $i$ is  odd and $\alpha_{i-1}>\alpha_{i}$ then $\psi(i)=1$;
\item If $i$ is even and $\alpha_i>\alpha_{i+1}$ then $\psi(i)=-1$;
\item In all other cases $\psi(i)=0$.
\end{enumerate}

This satisfies some obvious properties.

For $\ell\in\BZ$ define $\kappa_\ell\in\{0,1\}$ by $(-1)^\ell=(-1)^{\kappa_\ell}$. Then

\begin{enumerate}
	\setcounter{enumi}{3}
\item $\psi_\alpha(1)=1$; $\psi_\alpha(\ell)=-1$ if $\ell$ is even;
\item If $k\le\ell$ is odd then $\sum_{i=1}^k\psi_\alpha(i)=1$;
\item If $k\le\ell$ is even then $\sum_{i=1}^k\psi_\alpha(i)=1+\psi_\alpha(k)$;
\item $\sum_{i=1}^\ell\psi_\alpha(i)=\kappa_\ell$.
\end{enumerate}

Suppose $\alpha$ is a partition of $n$, and all entries $\alpha_i$ of  $\alpha$ are even. Then let
$$
\alpha+\psi_\alpha=(\alpha_1+\psi_\alpha(1),\dots, \alpha_\ell+\psi_\alpha(\ell)).
$$
This is a partition of $n+\kappa_\ell$. Explicitly,
if we write $\alpha+\psi_\alpha=(\alpha_1',\dots,\alpha_\ell')$, and
if $2i+1\le\ell$, then
$$
(\alpha'_{2i},\alpha'_{2i+1})=
\begin{cases}
(\alpha_{2i},\alpha_{2i+1}),& \text{ if } \alpha_{2i}=\alpha_{2i+1};\\
(\alpha_{2i}-1,\alpha_{2i+1}+1),& \text{ if } \alpha_{2i}>\alpha_{2i+1}.
\end{cases}
$$
Also $\alpha'_1=\alpha_1+1$ and, if $\ell$ is even, $\alpha'_\ell=\alpha_\ell-1$.

\subsection{Parametrization of elliptic conjugacy classes}
\label{s:ellipticparam}
For classical groups the elliptic conjugacy classes in $W$ are
parametrized as follows.  Recall $\CP(n)$ is the set of partitions of
$n$,  $\CP(n)_0$ is the subset of $\CP(n)$ consisting of partitions
with an even number of parts, and $\CP(n)^{\text{odd}}$ is the set of
partitions of $n$ consisting of only odd parts.

\begin{enumerate}
	\item $G=\GLdagger(n)$. Then $G$ has two connected components: $G^0=GL(n)$ and $D=G\backslash G^0$. We have
	\begin{itemize}
		\item $[W^{G^0}_e]$ is a singleton, and the only element is the conjugacy class of Coxeter elements;
		\item $\WDec$ is parametrized by $\CP^{odd}(n)$.
	\end{itemize}
	\item $G=SO(2n+1)$ or $Sp(2n)$: $[W_e]$ is parametrized by $\CP(n)$.
	\item $G=O(2n)$. Then $G$ has two connected components: $G^0=SO(2n)$ and $D=G\backslash G^0$.

          \begin{itemize}
\item            $[W_e]$ is parametrized by $\CP(n)$;
		\item $[W^{G^0}_e]$ is parametrized by $\CP(n)_0$;
		\item $\WDec$ is parametrized by $\CP(n)_1$.
	\end{itemize}
\end{enumerate}
For $\alpha\in\CP(n)$, we write $\CC_\alpha$ for the corresponding elliptic conjugacy class in $W$.

\subsection{Explicit description of $\Phi$ for classical groups}\label{explicit}
Suppose $G$ is a classical group. Lusztig's map $\Phi_e: \WDec \to [D_u]$ is described explicitly in
\cite[\S 4.2]{L1} and \cite[\S 3.7 \& \S 5.5]{L3}.

\begin{enumerate}
\item $G=\GLdagger(n)$. In this case, $\Phi_e$ sends the conjugacy class of Coxeter elements to the principal unipotent class. Note that both the conjugacy class of Coxeter elements and the principal unipotent class correspond to the partition $(n)$ of $n$.

If $p=2$, then $D$ contains unipotent elements. In this case, the map $\Phi_e: \WDec \to [D_u]$ is given by $\CC_\alpha\mapsto \CU_{\alpha,\epsilonmax^\alpha}$ for $\a \in \CP^{odd}(n)$.

\item $G=O(2n+1)$. The elliptic conjugacy classes of $W$ are parametrized by $\CP(n)$.

\begin{itemize}
\item[(i)] $p \neq 2$: the map $\Phi_e$ is given by $\CC_{\alpha} \mapsto \CU_{\alpha'}$, where
  $$
  \alpha'=\begin{cases}
2\alpha+\psi_{\alpha},    & \text{ if $\alpha$ has an odd number of parts };
\\ (2\alpha+\psi_{\alpha}, 1), & \text{ if $\alpha$ has an even number of parts}.
\end{cases}
$$
\item[(ii)] $p=2$: the map $\Phi_e$ is given by $\CC_{\alpha} \mapsto \CU_{\alpha', \epsilonmax^{\alpha'}}$
  where $\alpha'=(2\alpha,1)$.
\end{itemize}

\item $G=Sp(2n)$. The elliptic conjugacy classes of $W$ are parametrized by $\CP(n)$.
\begin{itemize}
\item[(i)]  $p \neq 2$: the map $\Phi_e$ is given by $\CC_{\alpha} \mapsto \CU_{2 \alpha}$.
\item[(ii)]  $p=2$: the map $\Phi_e$ is given by $\CC_{\alpha} \mapsto \CU_{(2 \alpha,\epsilonmax^{2\alpha})}$.
\end{itemize}

\item $G=O(2n)$.

If $p \neq 2$, then the unipotent elements of $G$ are contained in
$G^0=SO(2n)$. In this case, $[W^{G^0}]$ is parametrized by
$\CP(n)_0$. The map $\Phi_e: [W^{G^0}_e] \to [G^0_u]$ is given by
$\CC_{\alpha} \mapsto \CU_{2 \alpha+\psi_{ \alpha}}$.

If $p=2$, then $[W_e]$ is parametrized by $\CP(n)$ and the map
$\Phi_e: [W_e] \to [G_u]$ is given by
$\CC_{\alpha}\mapsto \CU_{2\alpha,\epsilonmax^{2\alpha}}$.
On the other hand $\WDec$ is parametrized by $\CP(n)\backslash \CP(n)_0$,
and the map $\Phi_e:\WDec\rightarrow \Du$ is $\CC_\alpha\mapsto\CU_{2\alpha,\epsilonmax^{2\alpha}}$.

Note that each unipotent class $\CU_{2\alpha,\epsilonmax^{2\alpha}}$ is a  single $G^0$-conjugacy classes.
\end{enumerate}

\subsection{Map from characteristic $0$ to characteristic $2$}
For the moment let $G_p$ be a connected reductive group, defined over an
algebraically closed field $\BF$ of characteristic $p>0$. Let $G_0$
be the complex group with the same root datum as $G$.  Consider the
sets $[G_{p, u}]$ and  $[G_{0, u}]$) of unipotent conjugacy classes of $G_p$
and $G_0$, respectively.

\begin{proposition}{\cite{Spa}*{Theorem III.5.2}}
\label{p:p}
  There is an injective, dimension preserving map $\pi_p:[G_{0, u}]\rightarrow [G_{p, u}]$, such that $\pi_p$ is an isomorphism of partially ordered
  sets from $[G_{0, u}]$ to its image.
\end{proposition}

Lusztig gives an alternative description of this map in terms of the Springer correspondence \cite{L1}*{\S 4.1},
and shows 
in \cite[Theorem 0.4 and \S3.9]{L2} that the following diagram is commutative:
\[
\xymatrix{& [W] \ar[ld]_{\Phi_0} \ar[rd]^{\Phi_p} & \\ [G_{0, u}] \ar@{^{(}->}[rr]^{\pi_p} & & [G_{p, u}].}
\]

We now assume $G$ is a classical group.
In this case 
$\pi_p$ is a bijection for $p\neq 2$. In the table below, we list the sets parametrizing the objects in types $B,C,D$, over $\C$ and in characteristic $2$.
\[
\begin{tabular}{lll}
  $G$ & $\C$ & $\Fbartwo$\\\hline
  $Sp(2n)$ & $\CP_{-1}(2n)$ & $\tPminus(2n)$ \\
  $O(2n+1)$ & $\CP_{1}(2n+1)$ & $\tPminus(2n)$\\
  $O(2n)$ & $\CP_{1}(2n)$ & $\tPminus(2n)$\\
\end{tabular}
\]

We have the following diagram
$$
\xymatrix{
  [G_{0,u}]\ar@/^/[r]^{\pi_2}& [G_{2,u}]\ar@/^/[l]^{\theta_2},
}
$$
where $\theta_2\circ\pi_2=\text{id}$.
Both $\pi_2,\theta_2$ are defined explicitly
in \cite{Spa}*{III, \S6-8}.

Implicit in the statement that Spaltenstein's and Lusztig's definitions of
$\pi_2$ agree is the following result.
This makes the relationship between Lusztig's function $\psi_\alpha$ and
Spaltenstein's map $\theta_2$ precise in the case of classical groups.
%\remind{maybe modify English}

\begin{lemma}
Suppose $G$ is of type $B,C$ or $D$, and $\CU\in\Phi_2\Wec$.
  Write $\CU=\CU_{(\alpha,\epsilonmax^{\alpha})}$ as in Section \ref{explicit}. Then
$$
\theta_2(\CU_{(\alpha,\epsilonmax^\alpha)})=
\begin{cases}
\CU_{\alpha}&\text{type }C  \\  
\CU_{\alpha+\psi_\alpha}&\text{types}\ B,D 
\end{cases}
$$
Thus the map $\theta_2$, which is inverse to $\pi_2$, has a simple description when restricted to the image of the elliptic elements.
%\remind{the last sentence is not easy to understand}
\end{lemma}

\begin{proof}
This is immediate in the case of $\Sp(2n)$: by \cite{Spa}*{III, 6.1} $\theta_2(\CU_{\alpha,\epsilonmax^\alpha})=\CU_\alpha$.
%\remind{Is it true that $\psi=0$ in this case? Is it how $\psi$ is defined?}

Suppose $G=\SO(2n)$ and $\beta\in\CP(n)$. Then $\Phi(\CC_\beta)=\CU_{(2\beta,\epsilonmax^{2\beta})}$.
So suppose $\alpha\in\CP(2n)$ has all even parts.

Define $\delta$ by \cite{Spa}*{III, Lemma 7.3}, so that 
$$
\theta_2(\CU_{(\alpha,\epsilonmax^\alpha)})=\CU_{\delta(\alpha)}
$$
Set $\gamma=\alpha+\psi_\alpha$, so we  have to show $\delta=\gamma$.

According to  \cite{Spa}*{Lemma 7.3}, and using
the definition of $\epsilonmax^\alpha$, we see $\epsilonmax^\alpha(i)=1$ if and only if
there is a row of length $i$, i.e. $m_\alpha(i)>0$.
Given this we see
\begin{subequations}
\renewcommand{\theequation}{\theparentequation)(\alph{equation}}
\begin{align}
\delta_i^*=\alpha_i^*+1, & \quad\text{if $i$ is odd, $\alpha_i^*$ is even, $m_\alpha(i-1)>0$};\\
\delta_i^*=\alpha_i^*-1, & \quad\text{if $i$ is even, $\alpha_i^*$ is even, $m_\alpha(i)>0$};\\
\delta_i^*=\alpha_i, & \quad\text{otherwise}.
\end{align}
\end{subequations}
On the other hand
\begin{subequations}
\renewcommand{\theequation}{\theparentequation)(\alph{equation}}
\begin{align}
\gamma_j=\alpha_j+1, & \quad\text{if $j$ is odd, $\alpha_{j-1}>\alpha_j$};\\
\gamma_j=\alpha_j-1, & \quad\text{if $j$ is even, $\alpha_j>\alpha_{j+1}$};\\
\gamma_j=\alpha_j, & \quad\text{otherwise}.
\end{align}
\end{subequations}
There is a bijection:
$$
\{i\text{ odd}\mid \alpha_i^*\text{ is even},\, m_\alpha(i-1)> 0\}\leftrightarrow\{j\text{ odd}\mid \alpha_{j-1}>\alpha_j\}.
$$
The bijection takes $i\mapsto j=\alpha^*_i+1$, with inverse $j\mapsto i=\alpha_j+1$.
Under this bijection incrementing $\alpha_i^*$ by $1$ corresponding to incrementing $\alpha_j$ by $1$.
Similarly there is a bijection
$$
\{i\text{ even}\mid \alpha_i^*\text{ is even},\, m_\alpha(i)> 0\}\leftrightarrow\{j\text{ even}\mid \alpha_{j}>\alpha_{j+1}\},
$$
taking $i\mapsto j=\alpha_i^*$ and $j\mapsto i=\alpha_j$,
under which decreasing $\alpha_i$ by $1$ corresponds to decreasing $\alpha_j^*$ by $1$.
This completes the proof for $\SO(2n)$. We leave the very similar case of $\SO(2n+1)$ to the reader.
\end{proof}

\begin{example}
  Suppose $G=\SO(2n)$ and $\alpha=[6,6,4,2]$.
  Then $\alpha^*=[4,4,3,3,2,2]$. Thus $\delta_1^*=\alpha_1^*=4, \delta_2^*=\alpha_2^*-1=3,\delta_3^*=\alpha_3^*=3,
\delta_4^*=\alpha_4^*=3, \delta_5^*=\alpha_5^*+1=3,\delta_6^*=\alpha_6^*-1=1,
\delta_7^*=\alpha_7^*+1=1$. This gives the partition $\delta^*=[4,3,3,3,3,1,1]$, or $\delta=[7,5,5,1]$.

On the other hand
$\gamma_1=\alpha_1+1=7,\gamma_2=\alpha_2-1=5,\gamma_3=\alpha_3+1=5,\gamma_4=\alpha_4-1=1$,
so $\gamma=[7,5,5,1]=\delta$.

\end{example}

Returning to Proposition \ref{p:p}, this gives a reduction of Theorem \ref{main} in the classical case
to characteristic $2$, in which case the
explicit description of $\Phi_2$ restricted to $[W_e]$ is quite simple.
We have the following result.

\begin{proposition}
  \label{p:Phi_elliptic}
  Let $G$ be a connected classical group defined over an algebraically
  closed field of characteristic $p \ge 0$.  Let
  $\CC_\a, \CC_\b\in\Wec$, parametrized by certain partitions as in
  Section \ref{s:ellipticparam}.  Then
  $\Phi(\CC_\a) \leu \Phi(\CC_\b)$ if and only if $\a \le \b$.

  The same statement holds for $\GLdagger(n)$ and $\O(n)$, with  $\CC_\alpha,\CC_\beta\in\WDec$ and $\Phi_e:\WDec\mapsto\Du$.
\end{proposition}

\begin{proof}
If $p \neq 2$, then $[G_{p, u}]$ and $[G_{0, u}]$ are the same as partially ordered sets.
Moreover, the map $\pi_2: [G_{0, u}] \to [G_{2,u}]$ is an
isomorphism from the poset $[G_{0, u}]$ to its image. In other
words, $\Phi_p(\CC_\a) \leu \Phi_p(\CC_\b)$ if and only if
$\Phi_2(\CC_\a) \leu \Phi_2(\CC_\b)$, where $\Phi_p$ is Lusztig's
map for $G$ in characteristic $p$.

	Now by \S \ref{explicit}, for $G=Sp(2n)$ or $O(n)$, the map $\Phi_2$ is given by $\CC_{\alpha} \mapsto \CU_{\alpha', \epsilonmax^{\alpha'}}$
	where $\alpha'=2 \a$ or $(2\alpha,1)$. By \S \ref{s:closure}, $\CU_{\alpha', \epsilonmax^{\alpha'}} \le \CU_{\b', \epsilonmax^{\b'}}$ if and only if $\a' \le \b'$, which is equivalent to $\a \le \b$.

        In the twisted cases $\GLdagger(n)$ and $\O(2n)$ we only have to consider the case $p=2$, in which case $\Phi_2$ is given
        by $\alpha\mapsto (\alpha, \e^{\a}_{\max})$ or $(2\alpha,\epsilonmax^
        {2\alpha})$ respectively, and the same argument applies.
      \end{proof}

Using this Theorem \ref{t:main} for classical groups 
is  equivalent to the following
combinatorial statement purely  about the Weyl group.

\begin{proposition}
  \label{p:order_reversing}
  Suppose $W$ is a classical Weyl group.
  Suppose $\CC_\alpha,\CC_\beta\in \WDec$, for partitions $\alpha,\beta$ as in Section \ref{s:ellipticparam}.
  Then
\begin{equation}
\label{e:order_rreversing}
\CC_\alpha\leW\CC_\beta\Leftrightarrow \beta\le\alpha.
\end{equation}
\end{proposition}

\begin{proof}[Proof of Theorem \ref{t:main} for classical groups]
Then by Propositions \ref{p:Phi_elliptic} and \ref{p:order_reversing},
  $$
  \Phi(\CC_\alpha)\leu\Phi(\CC_\beta)\Leftrightarrow \alpha\le\beta\Leftrightarrow \CC_\beta\leW\CC_\alpha,
  $$
  which implies Theorem \ref{main} for classical groups.
\end{proof}      

We prove Proposition \ref{p:order_reversing} in Section \ref{s:classical_proof}.

\section{Examples}

We explicitly describe the maps from $\Wc$ to $[G(\C)_u]$ and $[G(\overline \BF_2)_u]$ for some classical groups of small rank. 

The elliptic conjugacy classes in $W$ are parametrized by partitions
as discussed above. General conjugacy classes in $W$ are parametrized by a
pair $(L,\CC_L)$ where $L$ is a Levi factor and $\CC_L$ is an elliptic
conjugacy class in $L$ \cite{geck_pfeiffer}.  When $L$ is of type $A$
it only has one elliptic conjugacy class so we drop it from the notation.

In each table the elliptic classes are listed first, followed by a
line, and then the non-elliptic classes.

The unipotent orbits in $G(\C)$ are parametrized by certain
partitions, and the unipotent orbits in $G(\Fbartwo)$ are
parametrized by pairs, as described in Section \ref{s:classical}.

\subsection{$\Sp(4)$}

Here $[W]\overset{\Phi_2}\longrightarrow[Sp(4,\Fbartwo)_u]$ is a bijection (both sets have $5$ elements),
whereas $[\Sp(4,\C)_u]$ has $4$ elements.
\bigskip\hfil

\begin{tabular}{lll}
  $[W]$ & $Sp(4,\C)$ & $\Sp(4,\overline\BF_2)$\\
  \hline\hline
  $[2]$ & $[4]$& $([4],*)$ \\
  $[1,1]$ & $[2,2]$& $([2,2],\epsilon(2)=1)$ \\
  \hline
  $A^s_1$ & $[2,2]$ &$([2,2], \epsilon(2)=0)$ \\
  $A^l_1$ & $[2,1,1]$& $([2,1,1], *)$\\
  $T$ & $[1,1,1,1]$ & $([1,1,1,1],*)$
\end{tabular}
\begin{comment}

\bigskip

$\O(4)$

The unipotent orbits in $\O(4,\overline \BF_2)$ are parametrized
by $\tPminus(4)$.

The Weyl group of $\O(4)$ is of type $C_2$ which has five conjugacy
classes, two of which are elliptic. We borrow the parametrization of
these conjugacy classes from $W(C_2)$.

\begin{tabular}{llll}
  $[W]$ & $O(4,\C)$ & $\O(4,\overline\BF_2)$& $\SO/\O$\\
    \hline\hline
  $[2]$ & $[3,1]$& $([4], *)$ & $\O$\\
  $[1,1]$ & $[2,2]$& $([2,2], \epsilon(2)=1)$ &$\SO$\\
  \hline
   $(A^s_1,*)$ & $[2,2]$ &$([2,2], \epsilon(2)=0)$&$\SO$\\
  $(A^l_1,*)$ & $[2,1,1]$& $([2,1,1],*)$&$\O$\\
  $(T,*)$ & $[1,1,1,1]$ & $([1,1,1,1],*)$ & $\SO$
\end{tabular}

\bigskip
\end{comment}

\subsection{$\SO(4)$}

\hfil

The unipotent orbits in $\SO(4,\overline \BF_2)$ are parametrized
by $\tPminus(4)_0$ (odd rows have even multiplicity, the first column has even length).
Also a strongly even partitions count twice provided $\epsilon(i)=0$ for all $i$;
these are denote with a subscript $I$ or $II$.

There is one elliptic class in $W$ and $4$ overall. In this case all three sets are in bijection.

\begin{tabular}{lll}
  $[W]$ & $SO(4,\C)$ & $\SO(4,\overline\BF_2)$\\
  \hline\hline
  $[1,1]$ & $[3,1]$& $([2,2], \epsilon(2)=1)$ \\
    \hline
  $(A_1,*)$ & $[2,2]_I$ &$([2,2], \epsilon(2)=0)_I$ \\

  $(A'_1,*)$ & $[2,2]_{II}$& $([2,2], \epsilon(2)=0)_{II}$ \\
  $(T,*)$ & $[1,1,1,1]$ & $([1,1,1,1],*)$
\end{tabular}

\subsection{$\SO(6)$}

\hfil

The Weyl group of $\SO(6)$ is of type $D_3$ which has $5$ conjugacy
classes, $1$ of which is elliptic.

\begin{tabular}{lll}
  $[W]$ & $SO(6,\C)$ & $\SO(6,\overline\BF_2)$\\
    \hline\hline
  $[2,1]$ & $[5,1]$& $([4,2], *)$\\
  \hline
  $(A_2,*)$ & $[3,3]$ & $([3,3],*)$\\
  $(D_2,*)$ & $[3,1,1,1]$& $([2,2,1,1],\epsilon(2)=1)$\\
  $(A_1,*)$ &  $[2,2,1,1]$& $([2,2,1,1],\epsilon(2)=0)$\\
  $(T,*)$ &  $ [1,1,1,1,1,1]$& $([1,1,1,1,1,1],*)$\\
\end{tabular}

\subsection{$\SO(8)$}

\hfil

The Weyl group of $\SO(8)$ is of type $D_3$, which has $13$ conjugacy
classes, $3$ of which are elliptic.
Both $\SO(8,\C)$ and $\SO(8,\Fbartwo)$ have $12$ unipotent classes.
Note that the orbits $[3,2,2,1]$ and $([2,2,2,2],\epsilon(2)=1)$ both occur twice in the image.

Note that $SO(8)$ has two non-conjugate $GL(4)=A_3$ Levi factors
in addition to $GL(1)\times SO(6)$ of type $D_3$.
Also it has
two non-conjugate $\GL(2)\times \GL(2)$ factors.

Also a unipotent orbit for $\O(8,\Fbartwo)$ splits into two
for $\SO(8,\Fbartwo)$ if and only if all parts are even, with even multiplicity, and all $\epsilon(i)=0$.

\begin{tabular}{lll}
$[W]$ & $SO(8,\C)$ & $\SO(8,\overline\BF_2)$\\
$[3,1]$ & $[7,1]$& $([6,2], *)$\\
$[2,2]$ & $[5,3]$& $([4,4], \epsilon(4)=1)$\\
$[1,1,1,1]$ & $[3,2,2,1]$& $([2,2,2,2], \epsilon(2)=1)$\\
  \hline
$D_3=SO(6)\times GL(1)$ & $[5,1,1,1]$& $[4,2,1,1]$\\
$A_3=GL(4)$ & $[4,4]_I$ & $([4,4],\epsilon(4)=0)_I$\\
$A'_3=GL(4)'$ & $[4,4]_{II}$ &$([4,4],\epsilon(4)=0)_{II}$\\
$A_1\times D_2=GL(2)\times SO(4)$ & $[3,2,2,1]$ & $([2,2,2,2],\epsilon(2)=1)$\\
$D_2=\GL(1)\times \SO(4)$ & $[3,1,1,1,1,1]$ & $([2,2,1,1,1,1],\epsilon(2)=1)$\\
$A_2=\GL(3)\times \GL(1)$&$[3,3,1,1]$ & $([3,3,1,1],*)$\\
$2A_1=\GL(2)\times \GL(2)$&$[2,2,2,2]_I$ & $([2,2,2,2],\epsilon(2)=0)_I$\\
$2A_1=\GL(2)\times \GL(2)'$&$[2,2,2,2]_{II}$& $([2,2,2,2],\epsilon(2)=0)_{II}$\\
$A_1=\GL(2)\times\GL(1)\times\GL(1)$&$[2,2,1,1,1,1]$ & $([2,2,1,1,1,1,1,1],\epsilon(2)=0)$\\
$T$ & $[1^8]$ & $([1^8],*)$
\end{tabular}

\subsection{$\SO(12)$}

\hfil

We only discuss the elliptic classes.

\begin{tabular}{llll}
$[W]$ & $SO(12,\C)$ & $\SO(12,\overline\BF_2)$\\
$[5,1]$ & $[11,1]$ & $([10,2],*)$\\
$[4,2]$ & $[9,3]$ & $([8,4],*)$\\
$[3,3]$ & $[7,5]$ & $([6,6],\epsilon(6)=1))$\\
$[2,2,1,1]$ & $[5,3,3,1]$ & $([4,4,2,2],\epsilon(4)=\epsilon(2)=1)$\\
$[1,1,1,1,1,1]$ & $[3,2,2,2,2,1]$ & $([2,2,2,2,2,2],\epsilon(2)=1)$
\end{tabular}

\section{Proof of the Main Theorem  for classical groups}
\label{s:classical_proof}
We start by giving a  version of the Bruhat order for classical groups which is convenient for
our purposes.

If $G=GL(n)$, then $W\simeq S_n$. For $w\in W$ and $1\le i,j\le n$, we define
$$
w\{i, j\}=\{1 \le k \le i\mid w(k) \ge j\}, \quad w[i,j]=|w\{i, j\}|. 
$$

If $G=SO(2n+1)$ or $Sp(2n)$ or $O(2n)$, then $W\cong S_n\ltimes(\BZ/2\BZ)^n$. We identify $W$ with the set of permutations $\s$ of $\{\pm 1, \pm 2, \ldots, \pm n\}$ satisfying $\s(-i)=-\s(i)$ for all
$i$. For $-n \le i, j \le n$, we define
$$
w\{i, j\}=\{-n \le k \le i, w(k) \ge j\}, \quad w[i,j]=|w\{i, j\}|. 
$$

We use the following labeling on the Dynkin diagrams. 

\begin{gather*}
B_n/C_n: \begin{dynkinDiagram}[backwards,arrows=false,root radius=.12cm,edge length=-1.0cm,labels={n,n-1,n-2,2,1}]{B}{o3.oo}
\end{dynkinDiagram} \\
D_n: \begin{dynkinDiagram}[backwards,arrows=false,root radius=.12cm,edge length=-1.0cm,labels={n,n-1,n-2,,1,2}]{D}{o3.ooo}\end{dynkinDiagram}
\end{gather*}

\begin{comment}
Finally suppose $W\cong S_n\ltimes(\BZ/2\BZ)^{n-1}$ is of type $D_n$, and
identify $W$ with the set of permutations $\s$ of
$\{\pm 1, \pm 2, \ldots, \pm n\}$ satisfying $\s(-i)=-\s(i)$ for all
$i$, and $|\{i\mid i\le-1, \sigma(i)\ge 1\}|$ is even. Define $w[i,j]$ as in type $B/C$.
\end{comment}

\begin{proposition}
  \label{p:bruhat}
  Suppose $W$ is simple and classical, and $x,y\in W$.

\begin{enumerate}
\item   If $G=GL(n)$, then $x\le y\Leftrightarrow x[i,j]\le y[i,j]$ for all $1\le i,j\le n$;
\item   If $G=SO(2n+1)$ or $Sp(2n)$,  then $x\le y\Leftrightarrow x[i,j]\le y[i,j]$ for all $-n\le i,j\le n$;
\item   If $G=O(2n)$, then $x\le y\Rightarrow x[i,j]\le y[i,j]$ for all $-n\le i,j\le n$.
\end{enumerate}
\end{proposition}

\begin{remark}
It is worth pointing out that although the Weyl group of $O(2n)$
is isomorphic to the Weyl group of $SO(2n+1)$ and $Sp(2n)$, the
simple reflections and thus the Bruhat order are different. In
particular, the explicit description of the Bruhat order for the
Weyl group of $O(2n)$ is more complicated than that of $SO(2n+1)$ or
$Sp(2n)$ and we only use a weak version  of the Bruhat order for
the Weyl group of $O(2n)$ here. The readers interested in the full
explicit description may refer to \cite{bb}.
\end{remark}

Part (1) is proved in \cite[Theorem 2.1.5]{bb}. Part (2) is proved in \cite[Theorem 8.1.8]{bb}. One may also argue as follows. In \cite{Lu-uneq}, Lusztig showed that the subgroup of a Coxeter group which is the fixed points of  the action of a diagram automorphism is again a Coxeter group and the Bruhat order of these two groups are compatible. Thus part (2) may be deduced directly part (1).  As to part (3), if $x, y \in W(SO(2n))$, then the statement follows from \cite[Theorem 8.2.8]{bb}. The general case follows from part (2) and the following Lemma. 

\begin{lemma}
We identify the (extended) Weyl group of $O(2n)$ with the Weyl group of $Sp(2n)$ as the abstract group $W=S_n \ltimes (\BZ/2 \BZ)^n$ in the natural way. We denote by $\le^D$ and $\le^C$ the Bruhat order for $W(O(2n))$ and $W(Sp(2n))$ respectively. Let $x, y \in W$. If $x \le^D y$, then $x \le^B y$. 
\end{lemma}

\begin{proof}
It suffices to consider the case where $y=x s_\a$ for some positive root $\a$ of $SO(2n)$. In this case, $x \le^D y$ is equivalent to say that $x(\a)$ is a positive root of $SO(2n)$. Under the identification of $W(O(2n))$ with $W(Sp(2n))$ as abstract groups, we may identify the set of positive roots of $SO(2n)$ as a subset of the set of positive roots of $Sp(2n))$. In particular, $x(\a)$ is also a positive root of $Sp(2n)$. Thus $x \le^B y$. 
\end{proof}

\smallskip

The remainder of this section is concerned with the proof of Proposition \ref{p:order_reversing}.

Following \cite{He07}\footnote{Note that the convention for the labeling on the simple reflections here is opposite to the labeling used in \cite{He07} and the formulas below are modified accordingly.}, for $1 \le a, b \le n$, we define $$s_{[a, b]}=\begin{cases} s_a s_{a+1} \cdots s_b, & \text{ if } a \le b; \\ 1, & \text{ otherwise}.\end{cases}$$

\subsection{Case 1: Type $B_n/C_n/D_n$}

Let $\alpha=(\alpha_1, \alpha_2, \ldots, \alpha_\ell)$ be a partition of $n$. The corresponding class $\CC_\alpha\in\Wec$ consists of the permutations
$w$ satisfying:
\begin{enumerate}
	\item There exists a decomposition $\{1, 2, \ldots, n\}=I_1 \sqcup I_2 \sqcup \ldots \sqcup I_l$;
	
	\item $|I_j|=a_j$ for all $j$;
	
	\item The orbits of $w$ on $\{\pm 1, \ldots, \pm n\}$ are $I_j \sqcup -I_j$ for $1 \le j \le l$.
\end{enumerate}

We have the following useful inequalities for the elements in $\CC_{\alpha}$.

\begin{proposition}\label{B-ineq}
  Suppose $w \in \CC_{\alpha}$ and $0 \le m \le n$. Then
$$
w[n-m, n-m+1] \ge \min\{k; a_1+\cdots+a_k \ge m\}.
$$
\end{proposition}

\begin{proof}
Let $I_1, \ldots, I_l$ be the subsets of $\{1, \ldots, n\}$
satisfying the conditions (1)-(3) for $w$. Let $1 \le j \le l$ and
$r=\max\{i\mid i\in I_j\}$.  Then $w^{|I_j|}(-r)=r$. Therefore
if $r \ge n-m+1$, then there exists $s \in \BN$ such that
$w^s(-r) \le n-m$ and $w^{s+1}(-r) \ge n-m+1$. In other
words,
$$
w\{n-m, n-m+1\} \cap (I_j \sqcup -I_j) \neq \emptyset.
$$
Therefore
$$
w[n-m, n-m+1] \ge |\{j; \max I_j \ge n-m+1\}|.
$$

Let $J=\{j; \max I_j \ge n-m+1\}$. Note that there are exactly $m$ elements in $I_1 \sqcup \ldots \sqcup I_l$ that are larger than or equal to $n-m+1$. We have $\sum_{j \in J} | I_j| \ge m$. As $|I_k|=a_k$ for all $k$ and $a_1 \ge \cdots \ge a_l$, we have $a_1+\cdots+a_{| J|} \ge \sum_{j \in J} | I_j| \ge m$. In other words, $| J| \ge \min\{k; a_1+\cdots+a_k \ge m\}$. The proposition is proved.
\end{proof}

\begin{proof}[Proof of Proposition \ref{p:order_reversing} for type B/C]

For any partition $\alpha=(a_1, \ldots, a_l)$ of $n$, we define $$w_{\alpha}=(s_{[2, n+1-a_1]} \i s_{[1, n]}) (s_{[2, n+1-a_1-a_2]} \i s_{[1, n-a_1]}) \cdots (s_{[1, a_l]}).$$ By \cite[Lemma 7.15]{He07}, $w_{\alpha}$ is a minimal length element in $\CC_{\alpha}$.

If $\alpha \ge \beta$, then
  $$
  \alpha_1 \ge \beta_1, \alpha_1+\alpha_2 \ge \beta_1+\beta_2, \ldots.
  $$
Thus \begin{gather*} s_{[2, n+1-\alpha_1]} \i s_{[1, n]} \le s_{[2, n+1-\beta_1]} \i s_{[1, n]}, \\ s_{[2, n+1-\alpha_1-\alpha_2]} \i s_{[1, n-\alpha_1]} \le s_{[2, n+1-\beta_1-\beta_2]} \i s_{[1, n-\beta_1]}, \\ \ldots. \end{gather*} So $w_{\alpha} \le w_{\beta}$ and thus $\CC_{\alpha} \leW\CC_{\beta}$.
	
On the other hand, if $\CC_{\alpha} \leW\CC_{\beta}$, then
there exists $w \in\CC_{\alpha}$ with $w \le w_{\beta}$. Let
$m=\beta_1+\cdots+\beta_k$. By direct computation, we have
$w_{\beta}[n-m, n-m+1]=k$. By Proposition
\ref{p:bruhat}, $w[n-m, n-m+1] \le k$. By Proposition
\ref{p:bruhat}(2), $k \ge \min\{k'; \alpha_1+\cdots+\alpha_{k'} \ge m\}$. In other
words, $\alpha_1+\cdots+\alpha_k \ge m=\beta_1+\cdots+\beta_k$. Therefore
$\alpha \ge \beta$.
\end{proof}

\begin{proof}[Proof of Proposition \ref{p:order_reversing} for type D]
%Let $W$ be the Weyl group of type $D_n$.

Let $\d$ be the permutation of $\{\pm1, \dots, \pm n\}$ such that $\d(1) = -1$ and $\d(k) = k$ for $2 \le k \le n$. Then $W=W^0 \sqcup W^0 \d$. For any partition $\a$, the corresponding conjugacy class $\CC_\a$ is contained in $W^0$ if $\a_1^*$ is even and is contained in $W^0 \d$ if $\a_1^*$ is odd. 

For $0 \le a<b \le n$, define 
\[
w_{a, b}=
\begin{cases}
s_{[3, n+1-b]} \i s_{[1, n-a]}, & \text{ if } b \le n-1;\\
1, & \text{ otherwise}. 
\end{cases}
\]

For any partition $\alpha=(a_1, \ldots, a_l)$ of $n$, we define $$w'_{\alpha}=\begin{cases} w_{[0, \a_1]} w_{[\a_1, \a_1+\a_2]} \cdots w_{\sum_{1 \le i \le l-1} \a_i, n}, & \text{ if } 2 \mid \a_1^*; \\ w_{[0, \a_1]} w_{[\a_1, \a_1+\a_2]} \cdots w_{\sum_{1 \le i \le l-1} \a_i, n} \d, & \text{ if } 2 \nmid \a_1^*.\end{cases}$$ By \cite[Lemma 7.19]{He07}, $w'_{\alpha}$ is a minimal length element in $\CC_{\alpha}$.

Similar to the proof of the Proposition in type $B/C$, we have $w'_\a \le w'_\b$ if $\a \ge \b$.

On the other hand, if $\CC_{\alpha} \leW \CC_{\beta}$, then there exists $w \in \CC_{\alpha}$ with $w \leq w'_{\beta}$. Let $m=\beta_1+\cdots+\beta_k$. By direct computation, we have $w'_{\beta} [n-m, n-m+1]=k$. By Proposition \ref{p:bruhat}(3) $w[n-m, n-m+1] \le k$. By Proposition \ref{B-ineq}, $k \ge \min\{k'; \alpha_1+\cdots+\alpha_{k'} \ge m\}$. In other words, $\alpha_1+\cdots+\alpha_k \ge m=\beta_1+\cdots+\beta_k$. Therefore $\alpha \ge \beta$.
\end{proof}

\subsection{Case 2: Type ${}^2 A_{n-1}$} Let $W=S_n$ be the group of permutations of $\{1, 2, \cdots, n\}$. For $1 \le i, j \le n$, we define $$w\{i, j\}=\{-n \le k \le n; k \le i, w(k) \ge j\},  \quad w[i,j]=|w\{i, j\}|.$$

\begin{comment}
  We have the following explicit description of the Bruhat order. \remind{reference}

\begin{proposition}\label{A-i-j}
Let $x, y \in W$. Then $x \le y$ if and only if $|x[i, j]| \le |y[i, j]|$ for any $1 \le i, j \le n$.
\end{proposition}
\end{comment}

Let $\d=(1, n) (2, n-1) \cdots$ be the longest element of $W$. Then the conjugation action of $\d$ on $W$ induces a bijection on the set of simple reflections and is a length-preserving automorphism. The map $W \to W: w \mapsto w \d$ induces a bijection from the set of $\d$-twisted conjugacy class of $W$ to the set of ordinary conjugacy classes of $W$. Since $\d$ is the longest element of $W$, the map is order-reversing.

The $\d$-twisted elliptic conjugacy classes of $W$ are parametrized by the partitions of $n$ with odd parts. Let $\alpha=(\alpha_1, \alpha_2, \ldots, \alpha_l)$ be a partition of $n$ with odd parts, i.e. $\alpha_1 \ge \alpha_2 \ge \cdots \ge \alpha_l>0$ are odd positive integers and $\alpha_1+\cdots+\alpha_l=n$. Let $\CC_{\alpha}$ be the corresponding $\d$-twisted elliptic conjugacy class in $W$. This is the conjugacy class of $W$ consisting of permutations $w$ satisfying the following conditions:

\begin{enumerate}
	\item There exists a decomposition $\{1, 2, \ldots, n\}=I_1 \sqcup I_2 \sqcup \ldots \sqcup I_l$;
	
	\item $|I_j|=\alpha_j$ for all $j$;
	
	\item The orbits of $w \d$ on $\{1, 2, \ldots, n\}$ are $I_j$ for $1 \le j \le l$.
\end{enumerate}

We have the following useful inequalities for the elements in $\CC_{\alpha}$.

\begin{proposition}\label{A-ineq}
	Let $w \in\CC_{\alpha}$ and $1 \le m \le n-1$, we have $$w\d[\lceil \frac{m}{2}\rceil, n-\lfloor \frac{m}{2} \rfloor+1]+w\d[n-\lfloor \frac{m}{2} \rfloor, \lceil \frac{m}{2}\rceil+1] \le n-\min\{k; \alpha_1+\cdots+\alpha_k \ge m\}.$$
\end{proposition}

\begin{proof} Let $I=w\d\{\lceil \frac{m}{2}\rceil, n-\lfloor \frac{m}{2} \rfloor+1\}$ and $I'=w\d \{n-\lfloor \frac{m}{2} \rfloor, \lceil \frac{m}{2}\rceil+1\}$. If $k \in I$, then $k \le \lceil \frac{m}{2}\rceil<\lceil \frac{m}{2}\rceil+1$ and thus $(w \d)\i (k) \notin I'$. In other words, $$I \cap (w \d) (I')=\emptyset.$$ Similarly,  $$I \cap (w \d) \i (I')=\emptyset.$$
	
In particular, for any $w\d$-orbit $I_j$, we have \[\tag{a}| (I_j \cap I)|+| (I_j \cap I')|=| (I_j \cap I)|+| (I_j \cap (w \d)(I'))| \le|I_j|\]

We claim that

(b) If $|(I_j \cap I)|+|(I_j \cap I')|= |I_j|$, then $I_j \subset \{\lceil \frac{m}{2}\rceil+1, \lceil \frac{m}{2}\rceil+2, \ldots, n-\lfloor \frac{m}{2} \rfloor\}$.

Note that
$$
|I_j \cap I|+|I_j \cap w \d(I')|)=|I_j \cap I|+|I_j \cap (w \d) \i (I')|=|I_j|.
$$
Since
$I \cap (w \d)(I')=I \cap (w \d) \i(I')=\emptyset$, we have
$$(I_j \cap I) \sqcup (I_j \cap (w \d)(I'))=(I_j \cap I) \sqcup (I_j
\cap (w \d) \i(I'))=I_j.$$ Thus
$$(w \d) (I_j \cap I')=I_j \cap (w \d)(I')=I_j \cap (w \d) \i(I')=(w
\d) \i(I_j \cap I').$$ In other words, $I_j \cap I'$ is a subset of
$I_j$ that is stable under the action of $(w \d)^2$. Since the order
of the action of $w \d$ on $I_j$ equals to $|I_j|$, which is an
odd integer. Hence $I_j \cap I'$ is a $w \d$-stable subset of
$I_j$. As $w \d$ acts transitively on $I_j$, we have
$I_j \cap I'=\emptyset$ or $I_j$. Hence $I_j \subset I$ or
$I_j \subset I'$. However, as
$\lceil \frac{m}{2} \rceil<n-\lfloor \frac{m}{2} \rfloor+1$, if
$k \in I$, then $(w \d)(k) \notin I$. Thus $I_j \subset I'$. In other
words, for any $k \in I_j$, $k \ge \lceil \frac{m}{2} \rceil+1$ and
$k \le n-\lfloor \frac{m}{2} \rfloor$.

(b) is proved.

Let $J=\{j; I_j \not\subset \{\lceil \frac{m}{2}\rceil+1, \lceil \frac{m}{2}\rceil+2, \ldots, n-\lfloor \frac{m}{2} \rfloor\}\}$. By (a) and (b), we have that $$| J|+| J'|=\sum_j (|I_j \cap J|+|I_j \cap J'|) \le \sum_j |I_j|-|J|=n-|J|.$$ Note that there are exactly $m$ elements outside $\{\lceil \frac{m}{2}\rceil+1, \lceil \frac{m}{2}\rceil+2, \ldots, n-\lfloor \frac{m}{2} \rfloor\}$. We have $\sum_{j \in J} | I_j| \ge m$. As $| I_k|=\alpha_k$ for all $k$ and $\alpha_1 \ge \cdots \ge \alpha_l$, we have $\alpha_1+\cdots+\alpha_{|J|} \ge \sum_{j \in J} | I_j| \ge m$. In other words, $|J| \ge \min\{k; \alpha_1+\cdots+\alpha_k \ge m\}$. The proposition is proved.
\end{proof}

Now we prove the following result.

\begin{proof}[Proof of Proposition \ref{p:order_reversing} for type ${}^2 A_{n-1}$]

  \begin{comment}
  \begin{proposition}
	Let $\alpha, \beta$ be partitions of $n$ with odd parts. Then $\CC_{\alpha} \leW\CC_{\beta}$ if and only if $\alpha \ge \beta$.
\end{proposition}
  \end{comment}

%\begin{proof}
	The strategy is similar to the proof of the Proposition  in types B/C. 
	
	For any partition $\alpha=(a_1, \ldots, a_l)$ of $n$ consisting of only odd parts, we define $$w''_{\alpha}=s_{[1, n+1-\frac{a_1+1}{2}]} \i s_{[\frac{a_1+1}{2}+1, n+2-\frac{a_1+1}{2}-\frac{a_2+1}{2}]} \i \cdots s_{[\sum_{1 \le i \le l-1} \frac{a_i+1}{2}+1, n+l-\sum_{1 \le i \le l} \frac{a_i+1}{2}]} \d.$$ By \cite[Lemma 7.13]{He07}, $w''_{\alpha}$ is a minimal length element in $\CC_{\alpha}$.
	
	If $\alpha \ge \beta$, then $\alpha_1 \ge \beta_1$, $\alpha_1+\alpha_2 \ge \beta_1+\beta_2$, $\ldots$. By the explicit formula, $w''_{\alpha} \le w''_{\beta}$ and thus $\CC_{\alpha} \leW \CC_{\beta}$.
	
	On the other hand, if $\CC_{\alpha} \leW \CC_{\beta}$, then
        there exists $w \in \CC_{\alpha}$ with $w \le
        w''_{\beta}$. Hence $w w_0 \ge w''_{\beta} w_0$, where $w_0=(1 n) (2, n-1) \cdots$. Let
        $m=\beta_1+\cdots+\beta_k$. By direct computation, we
        have
        $$w_{\beta} w_0[\lceil \frac{m}{2}\rceil, n-\lfloor
        \frac{m}{2} \rfloor+1]+w_{\beta} w_0[n-\lfloor \frac{m}{2}
        \rfloor, \lceil \frac{m}{2}\rceil+1]=n-k.$$
	
	By Proposition \ref{p:bruhat}(1), $$w w_0[\lceil \frac{m}{2}\rceil, n-\lfloor \frac{m}{2} \rfloor+1]+w w_0[n-\lfloor \frac{m}{2} \rfloor, \lceil \frac{m}{2}\rceil+1]\ge n-k.$$ By Proposition \ref{A-ineq}, $k \ge \min\{k'; \alpha_1+\cdots+\alpha_{k'} \ge m\}$. In other words, $\alpha_1+\cdots+\alpha_k \ge m=\beta_1+\cdots+\beta_k$. Therefore $\alpha \ge \beta$.
\end{proof}

\section{Exceptional Groups}\label{6-exc}

We prove Theorem \ref{main} for exceptional groups by computer
computation, using the Atlas of Lie Groups and Representations
software \cite{atlas}.

Suppose $G$ is a simple, untwisted exceptional group.  We first list
representatives of the elliptic conjugacy classes in the Weyl group,
and we choose these representatives to be of minimal length.  We use
the representatives of \cite{geck_pfeiffer}*{Appendix B}, and we name
the classes according to these tables.  We explicitly compute the
order relation on these classes.

We label the unipotent classes in characteristic $0$ as in \cite{Spa}*{Section IV.2}.
Using the explicit description
of Lusztig's map $\Phi$ \cite{L1}*{Section 4.3}, and the closure relations
for unipotent orbits of \cite{Spa}*{Section IV.2}, we compute the Hasse diagram for the
image of $\Phi$.

In the twisted cases ${}^3\negthinspace D_4$ and
${}^2\negthinspace E_6$ we label the elliptic conjugacy classes as in
\cite{L3}.  This uses the label of the ordinary conjugacy class, from
the CHEVIE software in type ${}^3D_4$, and from \cite{carter} in
${}^2\negthinspace E_6$, and in both cases we also use ${}^!$ to
indicate the twisted class is elliptic.
The labelling of unipotent classes
classes is from \cite{Spa}, page 148 (${}^3D_4$) and page 250 (${}^2E_6$),
followed by (in brackets) 
that of \cite{M1} and \cite{M2}.
Lusztig defines the map $\Phi$  using  the latter notation.

%We give the table relating these labelings of unipotent classes after the diagrams.

In each case we give the bijection between elliptic conjugacy classes
and unipotent orbits and the corresponding Hasse diagram.

We explain the notation using the example of $G_2$.
\bigskip
\bigskip

\centerline{\includegraphics[scale=.1]{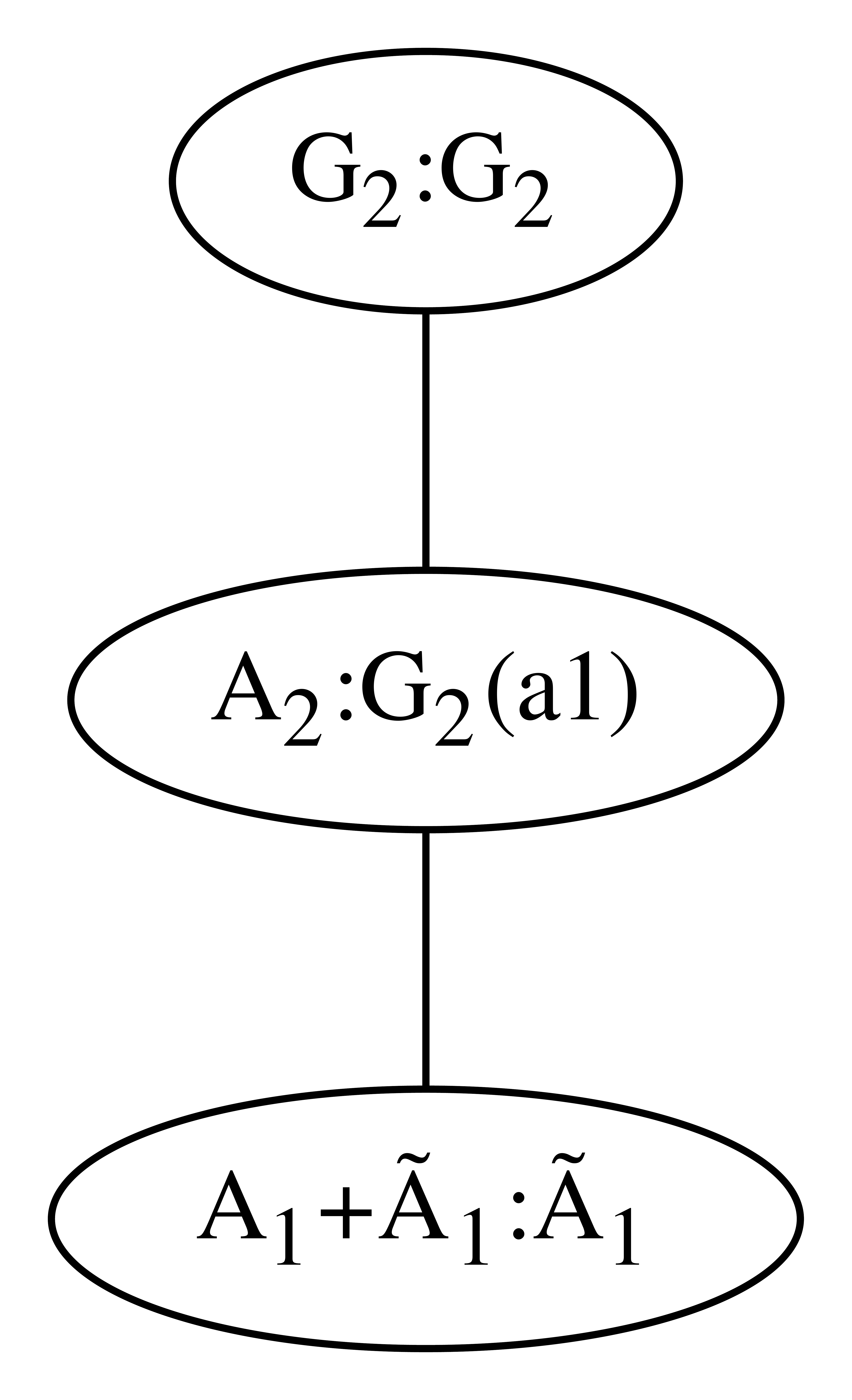}}

\bigskip
\bigskip

The left-hand entries $G_2,A_2$ and $A_1+\wt{A_1}$ are the elliptic
conjugacy classes of the Weyl group, of orders $6$ (the Coxeter element of $G_2$),
$3$ (the Coxeter element of $A_2$) and $2$ (the Coxeter element of $A_1+\wt{A_1}$), respectively.
The order relation is:
$$
G_2\leW A_2\leW A_1+\wt{A_1}.
$$
The entry following the colon
is a unipotent class in $G_2(\C)$, labeled as in \cite{Spa}; the image of
the corresponding elliptic conjugacy class under the Lusztig map.
These are $G_2,G_2(a_1)$ and $\wt{A_1}$, of dimension $12$ (the regular orbit), $10$ (the subregular orbit)
and $6$ respectively, of which the first two are distinguished. Here the order is the opposite one.
$$
\wt{A_1}\leu G_2(a_1)\leu G_2.
$$

\begin{figure}[h]
    \centering
    \begin{minipage}{0.50\textwidth}
        \centering
        \includegraphics[scale=.1]{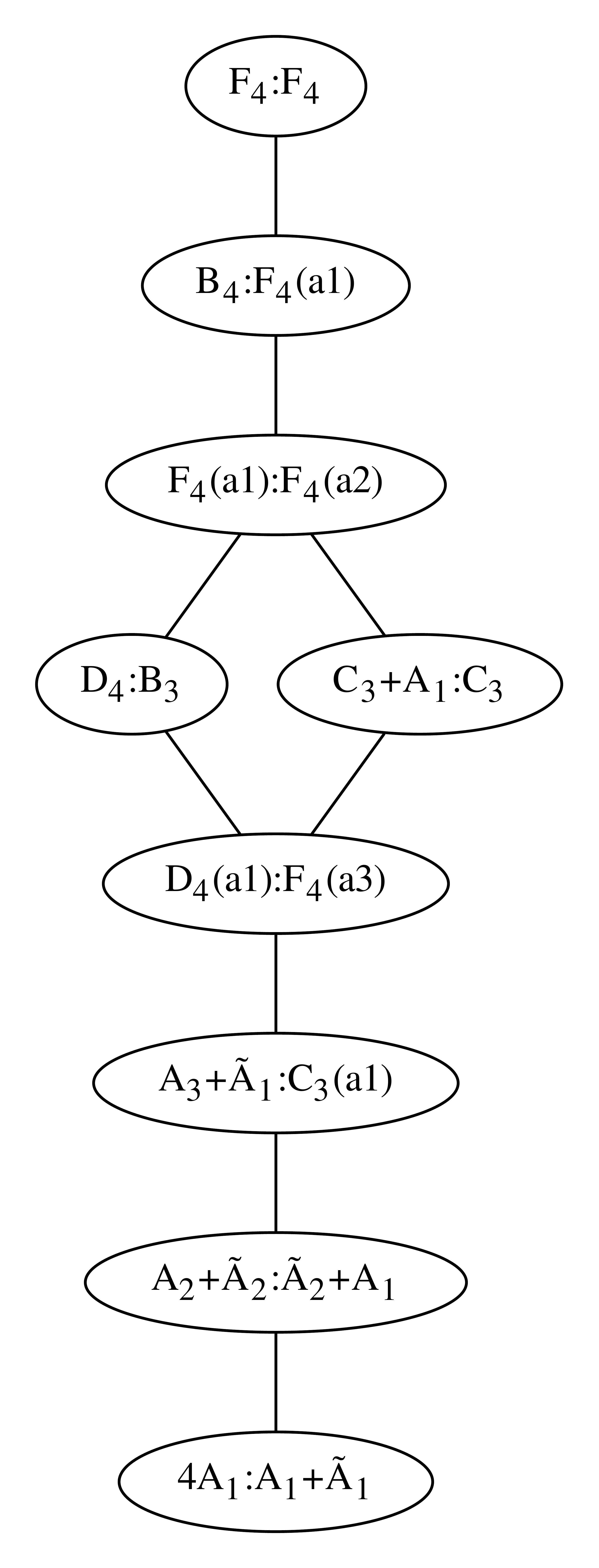} % first figure itself
        \captionsetup{labelformat=empty}
        \caption{Type $F_4$}
    \end{minipage}\hfill
    \begin{minipage}{0.45\textwidth}
        \centering
        \includegraphics[scale=.1]{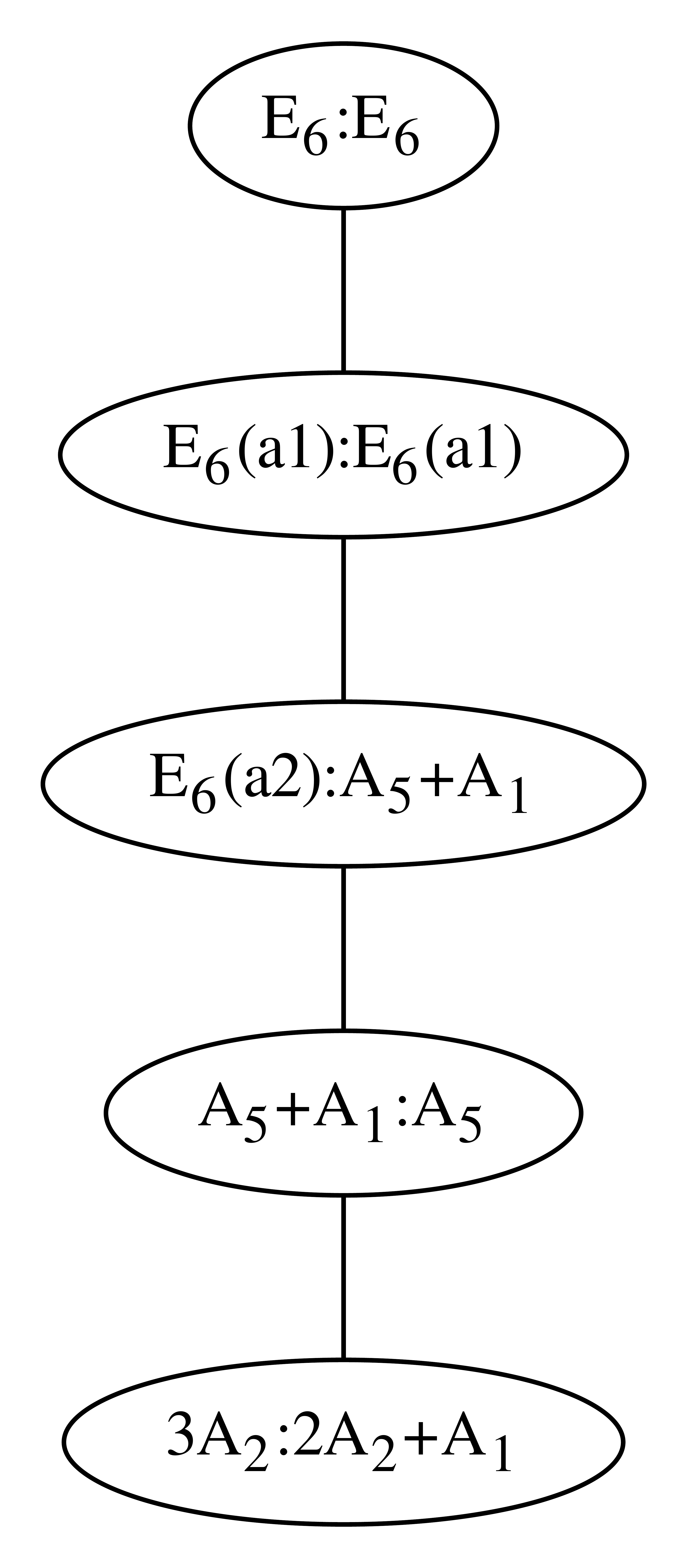} % first figure itself
                \captionsetup{labelformat=empty}
        \caption{Type $E_6$}
    \end{minipage}
  \end{figure}

\newpage

\begin{figure}
    \centering
    \begin{minipage}{0.45\textwidth}
        \centering
        \includegraphics[scale=.24]{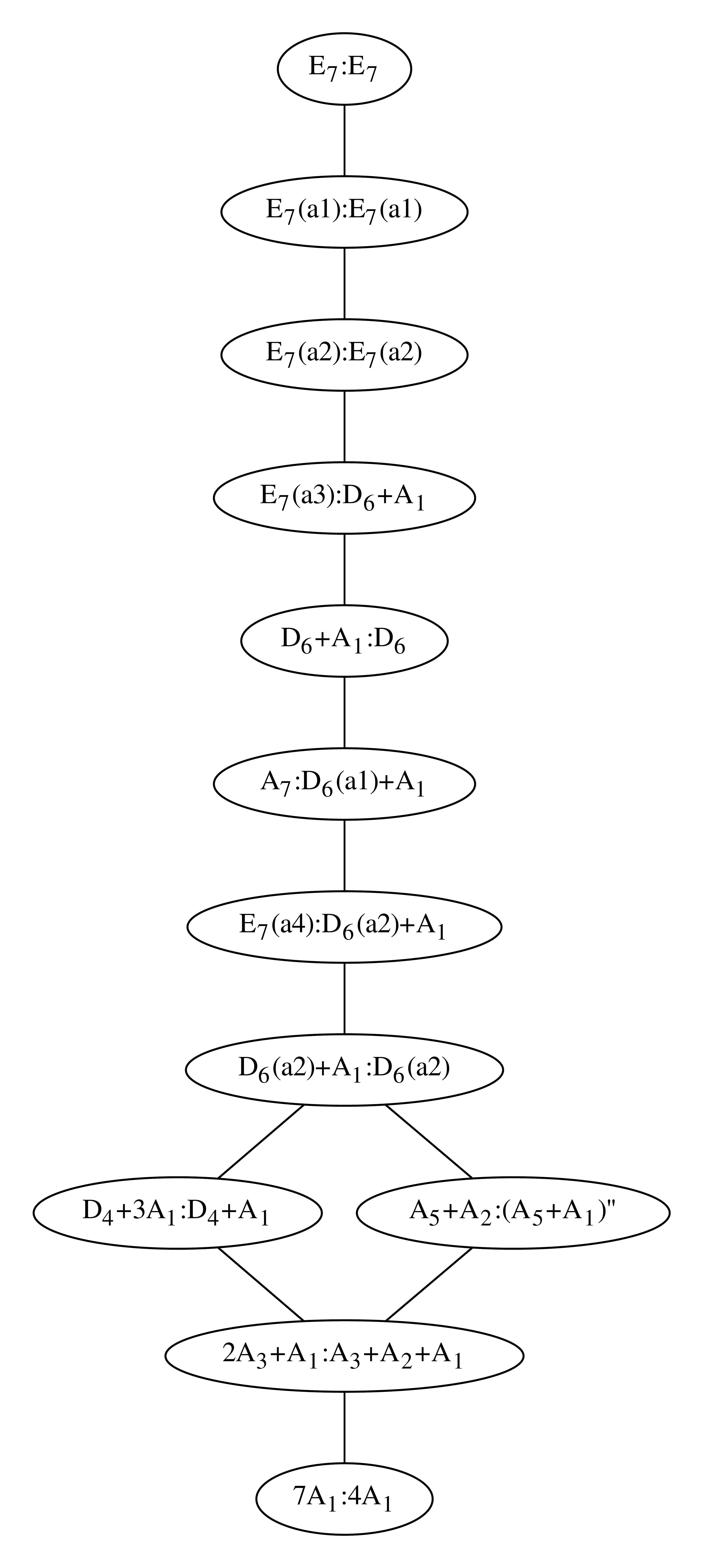} % first figure itself
        \captionsetup{labelformat=empty}
        \caption{\hskip .7in Type $E_7$}  %why are the captions not always centered?
    \end{minipage}\hfill
    \begin{minipage}{0.45\textwidth}
        \centering
        \includegraphics[scale=.28]{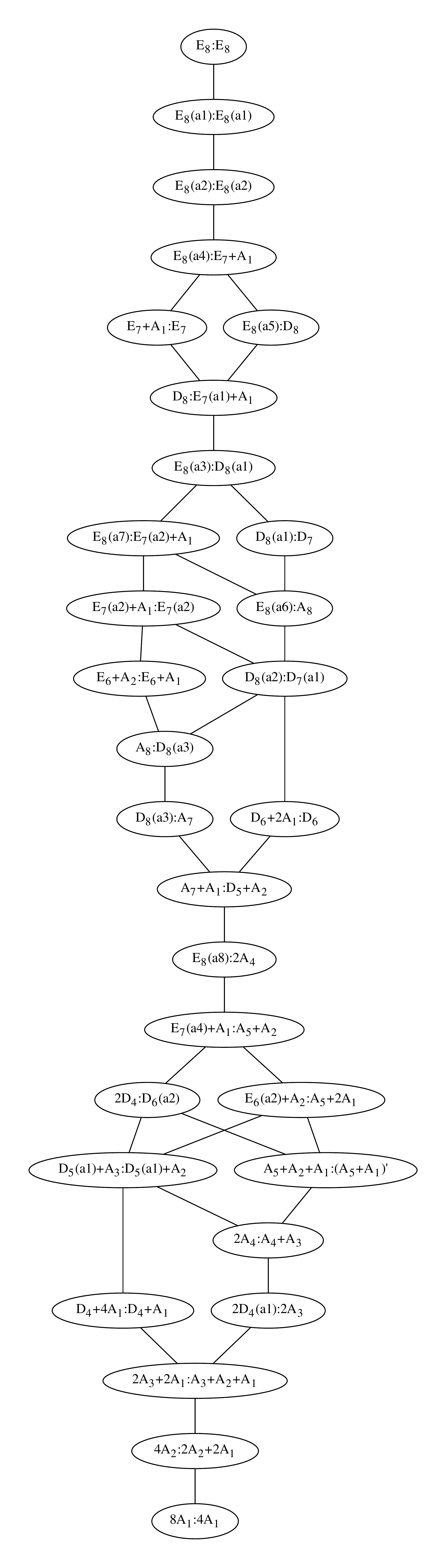} % first figure itself
                \captionsetup{labelformat=empty}
        \caption{\hskip -.3in Type $E_8$}
    \end{minipage}
  \end{figure}

\newpage

\begin{figure}[t]
    \centering
    \begin{minipage}{0.45\textwidth}
        \centering
        \includegraphics[scale=.1]{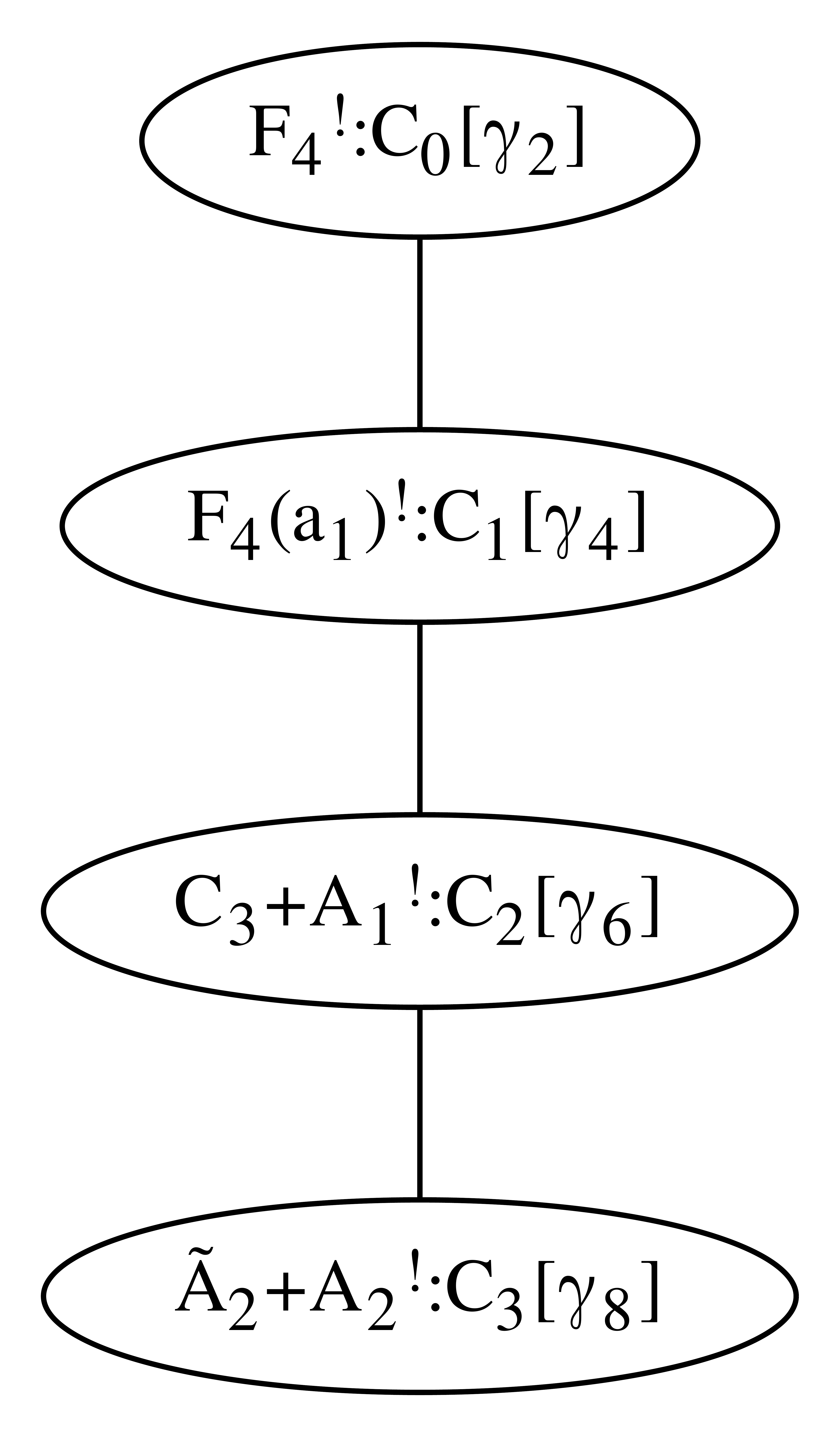} % first figure itself
        \captionsetup{labelformat=empty}
        \caption{Type $^3 D_4$}
    \end{minipage}\hfill
    \begin{minipage}{0.50\textwidth}
        \centering
        \includegraphics[scale=.19]{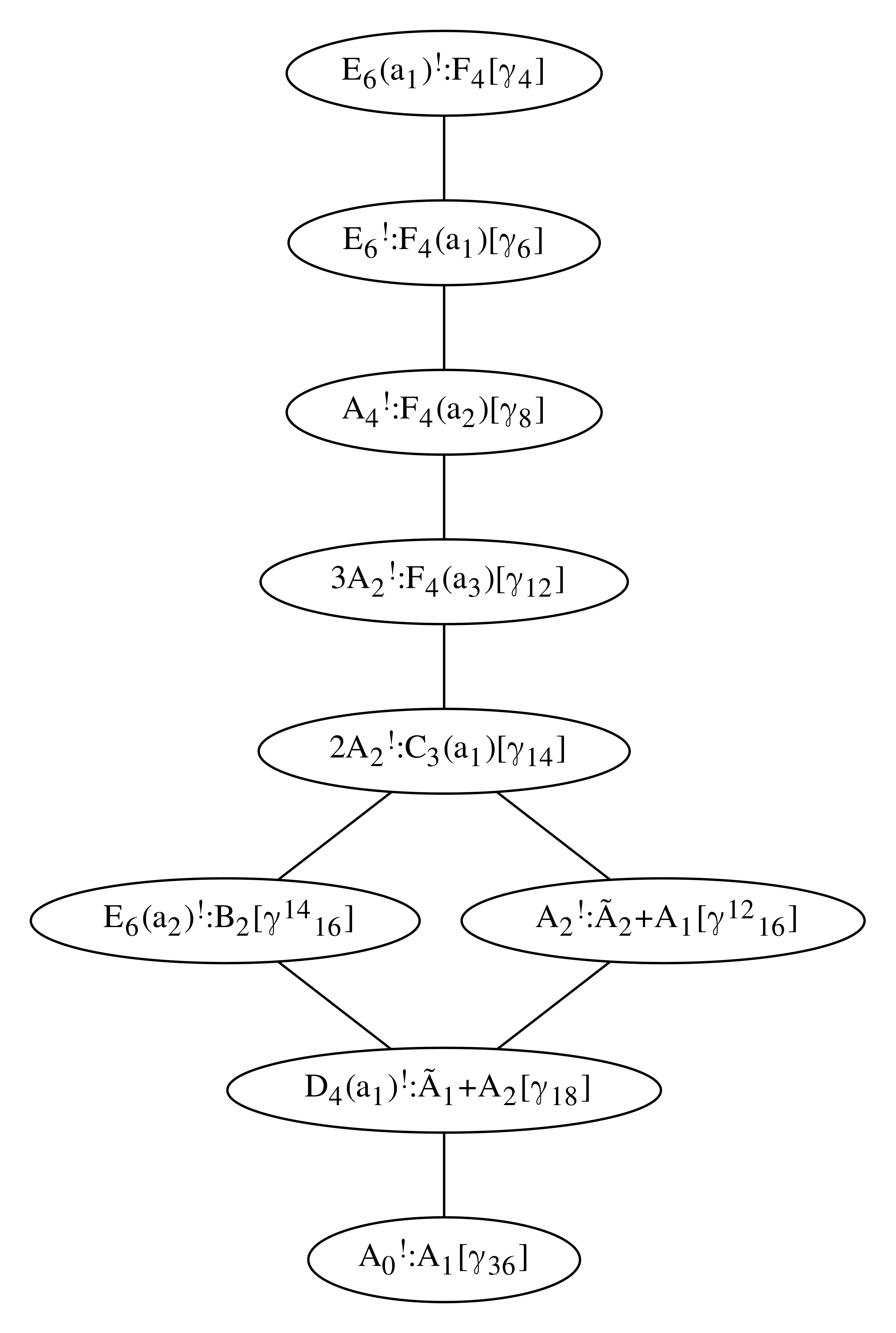}
                \captionsetup{labelformat=empty}
        \caption{\hskip 1in Type $^2E_6$}  
    \end{minipage}
  \end{figure}

\begin{comment}
  \begin{figure}[h]
    \centering
    \begin{minipage}{0.45\textwidth}
      \centering
                      \captionsetup{labelformat=empty}
                      \begin{tabular}{|c|c|}
                        \hline
    $C_0$ & $\gamma_{2}$ \\\hline
    $C_1$ & $\gamma_4$ \\\hline
    $C_2$ & $\gamma_6$ \\\hline
        $C_3$ & $\gamma_8$ \\\hline
    \end{tabular}
    \end{minipage}\hfill
    \begin{minipage}{0.45\textwidth}
      \hskip-1.1in

              \centering
                      \begin{tabular}{|c|c|}
                        \hline
    $F_4$ & $\gamma_{4}$ \\\hline
    $F_4(a_1)$ & $\gamma_6$ \\\hline
    $F_4(a_2)$ & $\gamma_8$ \\\hline
    $F_4(a_3)$ & $\gamma_{12}$ \\\hline
    $C_3(a_1)$ & $\gamma_{14}$ \\\hline
    $B_2$ & $\gamma_{16}^{14}$ \\\hline
    $\tilde A_2+A_1$ & $\gamma_{16}^{12}$ \\\hline
   $\tilde A_1+A_2$ & $\gamma_{18}$ \\\hline
   $A_1$ & $\gamma_{36}$ \\\hline
                      \end{tabular}

                \captionsetup{labelformat=empty}
              \end{minipage}

  \end{figure}

\end{comment}

%\includepdf[page={2,3,4}]{exceptionalDiagrams.pdf}

\end{document}